\documentclass[11pt]{amsart}

\usepackage{amsmath}
\usepackage{amssymb}
\usepackage{graphicx}
\usepackage{url}
\usepackage{multirow}

\usepackage{algorithm}
\usepackage{algorithmic}
\usepackage{float}
\usepackage{comment}
\usepackage[font=small,labelfont=bf]{caption}
\usepackage{subcaption}

\usepackage{color}
\usepackage{hyperref}

\newtheorem{theorem}{Theorem}[section]

\theoremstyle{definition}
\newtheorem{definition}[theorem]{Definition}
\newtheorem{example}[theorem]{Example}
\newtheorem{remark}[theorem]{Remark}

\newcommand{\moh}[1]{\textcolor{black}{#1}}

\newcommand{\alex}[1]{\textcolor{black}{#1}}
\newcommand{\rev}[1]{\textcolor{black}{#1}} % for the first reviewer
\newcommand{\rrev}[1]{\textcolor{black}{#1}} % for the second reviewer

\numberwithin{equation}{section}

\title[Adaptive Algorithms] {Adaptive Algorithms for Relatively Lipschitz Continuous Convex Optimization Problems}

\author[F. Stonyakin]{Fedor Stonyakin}
\address[F. Stonyakin]{Moscow Institute of Physics and Technology and V. I. Vernadsky Crimean Federal University, Russia}
\email{{\tt fedyor@mail.ru}}

\author[M. Alkousa]{Mohammad Alkousa}
\address[M. Alkousa]{Moscow Institute of Physics and Technology and HSE University, Russia}
\email{\tt mohammad.alkousa@phystech.edu}

\author[A. Titov]{Alexander Titov}
\address[A. Titov]{Moscow Institute of Physics and Technology and HSE University, Russia}
\email{\tt a.a.titov@phystech.edu}

\author[O. Savchuk]{Oleg Savchuk}
\address[O. Savchuk]{V. I. Vernadsky Crimean Federal University, Russia}
\email{\tt oleg.savchuk19@mail.ru}

\author[A. Gasnikov]{Alexander Gasnikov}
\address[A. Gasnikov]{Moscow Institute of Physics and Technology and IITP RAS, Moscow and Caucasus Mathematical Center, Adyghe State University, Maikop, Russia}
\email{\tt avgasnikov@gmail.com}

\thanks{The work of F. Stonyakin, A. Titov and O. Savchuk in Algorithms \ref{adaptive_alg4}, \ref{universal_alg2} and Theorems \ref{theorem_adaptive_Alg_2}, \ref{ThmUnivMeth2} was supported by Russian Science Foundation and Moscow (project No. 22-21-20065, \url{https://rscf.ru/project/22-21-20065}). The work of A. Gasnikov in Section \ref{experiments_Alg5} was supported by the Ministry of Science and Higher Education of the Russian Federation (Goszadaniye), No. 075-00337-20-03, project No. 0714-2020-0005.}



\begin{document}

\begin{abstract}
Recently, there were proposed some innovative convex optimization concepts, namely, relative smoothness \cite{Bauschke} and relative strong convexity \cite{Lu_Nesterov,Lu}. These approaches have significantly expanded the class of applicability of gradient-type methods with optimal estimates of the convergence rate. %, which are invariant regardless of the dimension of the problem. 
Later Yu. Nesterov and H. Lu \rrev{\cite{Lu_Nesterov,Lu}} introduced some modifications of the Mirror Descent method for convex minimization problems with the corresponding analogue of the Lipschitz condition (the so-called relative continuity or Lipschitz continuity). 
In this paper, we cover both the concept of relative smoothness and relative Lipschitz continuity and introduce some adaptive and universal methods which have optimal estimates of the convergence rate for the corresponding class of problems. %In this paper, we introduce the new concept of $(\alpha, L, \delta)$-relative smoothness which covers both the concept of relative smoothness and relative Lipschitz continuity. For the corresponding class of problems, we propose some adaptive and universal methods which have optimal estimates of the convergence rate. 
We consider the relative boundedness condition for the variational inequality problem and propose some adaptive optimal methods for this class of problems. Some results of the conducted numerical experiments are presented, which demonstrate the effectiveness of the proposed methods.
\end{abstract}

\maketitle

\section{Introduction}

The recent dramatic growth of various branches of science has led to the necessity of the development of numerical optimization methods in spaces of large and extra-large dimensions. A special place in modern optimization theory is given to gradient methods. Recently, there was introduced a new direction for the research, associated with the development of gradient-type methods for optimization problems with relatively smooth \cite{Bauschke} and relatively strongly convex \cite{Lu_Nesterov} functions. Such methods are in high demand and urgent due to numerous theoretical and applied problems. For example, the D-optimal design problem turned out to be relatively smooth \cite{Lu_Nesterov}. It is also quite interesting that in recent years there have appeared applications of these approaches (conditions of relative smoothness and strong convexity) to auxiliary problems for tensor methods for convex minimization problems of the second and higher orders \cite{Nest_tens,Nest_core}. It is worth noting that tensor methods make it possible to obtain optimal estimates of the rate of convergence of high-order methods for convex optimization problems \cite{kamzolov2022exploiting}.

A few years ago there was introduced a generalization of the Lipschitz condition for nonsmooth problems, namely, relative Lipschitz continuity \cite{Lu,Nestconf}. The concept of relative Lipschitz continuity essentially generalizes the classical Lipschitz condition and covers quite important applied problems, including the problem of finding the common point of ellipsoids (IEP), as well as the support vector machine (SVM) for the binary classification problem.

The concepts of relative smoothness, relative Lipschitz continuity, and relative strong convexity made it possible to significantly expand the limits of applicability of gradient type methods while preserving the optimal convergence rate $O (\frac{1}{\varepsilon^2})$ for relatively Lipschitz problems and $O(\frac{1}{\varepsilon})$ for relatively smooth problems ($\varepsilon$, as usual, denotes the accuracy of the solution \rrev{for functional residual}). The authors \rev{of} \cite{Dragomir} have shown that for the class of relatively smooth problems, such an estimate for the rate of convergence cannot be improved in the general case.

In this paper we consider the class of $(\alpha, L, \delta)$-relatively smooth objective functions (see Definition \ref{defalphacont}), which covers both the concept of relative smoothness and relative Lipschitz continuity. Let $Q$ be a closed convex subset of some finite-dimensional vector space. For the classical optimization problem 
\begin{equation}\label{main}
\min\limits_{x\in Q}f(x)
\end{equation}
we propose some analogues of the universal gradient method which automatically adjusts to the "degree of relative smoothness" of the $(\alpha, L, \delta)$-relatively smooth problem (Sect. \ref{sect_univers}). {\color{black}We also mention that the proposed algorithms are applicable to solve the problem of minimizing the relatively strongly convex functions, see \cite{savchuk2022adaptive} for more details. }
 
%{\color{red}We also consider the problem of minimizing the relatively strongly convex functions in two directions. The first one is the applicability of the introduced methods and their convergence rate, the second one is the restarted version of the Mirror Descent method for the constrained optimization problem.}

In addition to the classical optimization problem, we consider the problem of solving Minty variational inequality with  ($M$-)relatively bounded operator. For a given relatively bounded and monotone operator $g:Q\longrightarrow \mathbb{R}^n$, %which satisfies
%\begin{equation}\label{Rel_Boud}
    %\langle g(x),x-y\rangle \leqslant %M\sqrt{2V(y,x)} \quad \forall x,y \in %Q,
%\end{equation}
%for some $M>0$ ($V(y, x)$ is a Bregman divergence, defined in \eqref{Bregman}, Sect. \ref{basic}), and
%\begin{equation}\label{monotone_operator}
    %\langle g(y)-g(x),y-x\rangle \geqslant 0 \quad \forall x, y \in Q,
%\end{equation}
we need to find  a vector \rrev{$x_* \in Q$}, such that
\begin{equation}\label{VI_problem}
    \langle g(x),x_* - x\rangle \leqslant 0 \quad \forall x\in Q.
\end{equation}

Relative boundedness can be understood as an analogue of relative Lipschitz continuity for variational inequalities. It should be noted that the subgradient of a relatively Lipschitz continuous function satisfies the relative boundedness condition. This fact plays an important role in considering relatively Lipschitz continuous Lagrange saddle point problems and their reduction  to corresponding variational inequalities with the relatively bounded operator. Recently, in \cite{Stonyakin_etal} the authors proposed an adaptive version of the Mirror Prox method (extragradient type method) for variational inequalities with a condition similar to relative smoothness. It should be noted that variational inequalities with relatively smooth operators are applicable to the resource sharing problem \cite{Antonakopoulos}. Also, in \cite{Titov_etal} there were  introduced some non-adaptive switching subgradient algorithms for convex programming problems with relatively Lipschitz continuous functions. Recently, there was proposed a non-adaptive method for solving variational inequalities with the  relatively bounded operator \cite{MOTOR2021}. In this paper, we propose an adaptive algorithm for the corresponding class of problems.

The paper consists of the introduction and {\color{black} 6} main sections. In Sect. \ref{basic} we give some basic notations and definitions. In Sect. \ref{adapt} we consider the Minty variational inequality with a relatively bounded operator and propose an adaptive algorithm for \rrev{solving it}. Sect. \ref{adapt_univers} is devoted to adaptive algorithms for relatively smooth optimization problems. In Sect. \ref{sect_univers} we propose some universal algorithms for minimizing relatively smooth and relatively Lipschitz continuous functions. %{\color{black} In Sect. \ref{sect_online} we investigate  the applicability of the proposed methods to minimizing the relatively strongly convex functions.} 
Sect. \ref{experiments_Alg5} is devoted to the numerical experiments which demonstrate the effectiveness of the proposed methods.  

To sum it up, the contributions of the paper can be formulated as follows.
\begin{itemize}

\item We consider the variational inequality with the relatively bounded operator and propose some adaptive first-order methods to solve such a class of problems with optimal complexity estimates $O(\frac{1}{\varepsilon^2})$.
    
\item We introduce adaptive and universal algorithms for minimizing relatively smooth and relatively Lipschitz continuous functions and provide their theoretical justification. %an \\ \mbox{$(\alpha, L, \delta)$-relatively} smooth function and provide their theoretical justification. 
The stopping criteria of the introduced adaptive algorithms are \rrev{simple} (which is especially important in terms of numerical experiments), but universal algorithms are guaranteed to be applicable to a \rrev{wide} class of problems. %The class of  $(\alpha, L, \delta)$-relatively smooth functions covers both the relative smoothness and relative Lipschitz continuity concepts. 
%Moreover, 
Our approach allows us to minimize the sum of relatively smooth and relatively Lipschitz continuous functions, even though such a sum does not satisfy \rrev{neither relatively smoothness condition nor relatively Lipschitz one}. Theoretical estimates of the proposed methods are optimal both for convex relatively Lipschitz minimization problems $O(\frac{1}{\varepsilon^2})$ and convex relatively smooth minimization problems $O(\frac{1}{\varepsilon})$.
    
\item We provide the numerical experiments for the Intersection of   Ellipsoids  Problem  (IEP) and the Lagrange saddle point problem for the Support Vector  Machine (SVM) with inequality-type function constraints. We also, compare numerically, for (IEP), one of the proposed algorithms with the AdaMirr algorithm, which was recently proposed in \cite{AdaMirr_2021}. The conducted experiments demonstrate that the proposed algorithms work better than AdaMirr and they can work faster than the obtained theoretical estimates in practice. 
\end{itemize}

\section{Basic definitions and notations}\label{basic}

Let us give some basic definitions and notations concerning Bregman divergence and the prox structure, which will be used throughout the paper.

Let $(E,\|\cdot\|)$ be some normed \rrev{finite-dimensional real vector space} and $E^*$ be its  \rrev{dual space} with the norm
$$
    \|y\|_*=\max\limits_x\{\langle y,x\rangle,\|x\| \leqslant 1\},
$$
where $\langle y,x\rangle$ is the value of the \rrev{linear function} $y$ at $x \in E$. Assume that $Q\subset E$ is a closed convex set (for variational inequalities in Sect. \ref{adapt} we consider a convex compact set $Q\subset E$). 

Let $d: Q \longrightarrow \mathbb{R}$ be a distance-generating function (d.g.f) which is continuously differentiable and convex.
%Assume that $\min\limits_{x\in Q} d(x) = d(0)$ and suppose that there exists a constant $\Theta_0 >0 $, such that $d(x_{*}) \leqslant \Theta_0^2,$ where $x_*$ is a solution of the problem \eqref{main} (supposing that the problem \eqref{main} is solvable). Note that if there is a set of optimal points $X_* \subset Q$, we assume that $\min\limits_{x_* \in X_*} d(x_*) \leqslant \Theta_0^2$.

For all $x, y\in Q \subset E$, we consider the corresponding Bregman divergence
\begin{equation*}\label{Bregman}
    V(y, x) = d(y) - d(x) - \langle \nabla d(x), y-x \rangle.
\end{equation*}

Now we introduce the following concept of $(\alpha, L, \delta)$-relative smoothness which covers both the concept of relative smoothness and relative Lipschitz continuity. Further, we denote by $\nabla f$ an arbitrary subgradient of $f$.

\begin{definition}\label{defalphacont}
Let us call a convex function $f:Q\longrightarrow \mathbb{R}$ {\it $(\alpha, L, \delta)$-relatively smooth} for some $\alpha \in [0; 1]$, $L>0$ and $\delta>0$, if the following inequalities hold
\begin{equation}\label{eqalpha1relsm}
f(y) \leqslant f(x) + \langle \nabla f(x), y - x \rangle + LV(y, x) + L\alpha V(x,y) + \delta, \quad \rrev{\forall x, y \in Q,}
\end{equation}
\begin{equation}\label{eqalpha1relsm1}
\alpha\left(\langle \nabla f(x), y - x \rangle +  LV(y, x) + \delta\right) \geqslant 0 \;\;\; \forall x, y\in Q,
\end{equation}
for each subgradient $\nabla f(x)$ of $f(x)$.
\end{definition}

It is obvious that for $\alpha = 0$, $L>0$, and $\delta = 0$ one gets the well-known relative smoothness condition (often defined as $L$-relative smoothness, see \cite{Bauschke} for $\delta = 0$ and \cite{Stonyakin_etal} for the case of $\delta >0$). For $\alpha = 1$, $L = \frac{2M^2}{\varepsilon}$, and $\delta = \rev{\frac{\varepsilon}{4}} >0 $\rrev{, where $\varepsilon$ is arbitrary,} the inequalities \eqref{eqalpha1relsm} and \eqref{eqalpha1relsm1} follow from the condition of the relative Lipschitz continuity (also known as relative continuity or $M$-relative Lipschitz continuity), proposed recently in \cite{Lu,Nestconf} 

\begin{definition}\label{defrelLIP}
Convex function $f:Q\longrightarrow \mathbb{R}$ is called {\it $M$-relatively Lipschitz continuous} for some $M>0$, if the following inequality holds
\begin{equation*}
\langle \nabla f(x), y - x \rangle + M\sqrt{2V(y,x)} \geqslant 0 \quad \forall x, y \in Q.
\end{equation*}
\end{definition}

Indeed, for each $x, y \in Q$ we have $$\langle \nabla f(x), x - y \rangle \leqslant M\sqrt{2V(y, x)} \leqslant \frac{2M^2 }{\varepsilon} V(y, x) + \frac{\varepsilon}{4}.$$
Further, $$f(y) - f(x) \leqslant \langle \nabla f(y), y - x \rangle \leqslant M\sqrt{2V(x, y)} \leqslant \frac{2M^2 }{\varepsilon} V(x, y) + \frac{\varepsilon}{4}$$
and 
$$f(y) \leqslant  f(x) + \langle \nabla f(x), y - x \rangle + \frac{2M^2 }{\varepsilon} V(y, x) + \frac{2M^2 }{\varepsilon} V(x, y) + \frac{\varepsilon}{2}.$$

%(the proof of Theorem \ref{ThmUnivMeth2} in Appendix E)
So, each relatively Lipschitz continuous function $f$ satisfies \eqref{eqalpha1relsm} for large enough $L > 0$ and $\delta > 0$. 

It is worth mentioning that the sum of the relatively smooth function $f_1$ and relatively Lipschitz continuous convex function $f_2$ satisfies the \mbox{$(\alpha, L, \delta)$-relative} smoothness condition, if 
$$
f_1(y) \geqslant f_1(x) - rV(y, x) - q \quad \forall x, y \in Q,
$$
for some fixed $r, q > 0$, and the corresponding values $\alpha, L, \delta >0$ (this assumption can be understood as limiting the fast growth of $f_1$ \alex{and takes place, for example, when a function defined on a bounded set is bounded from below}). Generally, such a sum is neither relatively smooth nor relatively Lipschitz continuous function. 

\rev{Let us note the following important fact, which obviously follows from Lemma 3.2 from \cite{Stonyakin_etal} and plays a key role in first inequalities in the following proofs. According to this fact for each operator $g: Q \rightarrow \mathbb{R}^n$ 
$$y = \arg\min\limits_{x\in Q}\Big\{\langle g(z), x\rangle + \beta V(x,z)\Big\}, \quad \beta \geq 0, z \in Q
$$
we have
\begin{equation}\label{Lemma}
 \langle g(z), x \rangle + \beta V(x,z) \geq \langle g(z), y  \rangle + \beta V(y,z) + \beta V(x,y), \;\;\; \forall x, z\in Q.
\end{equation}}

%It mans that the considered algorithms for \mbox{$(\alpha, L, \delta)$-relative} smoothness condition can be used to minimize relatively Lipschitz continuous function with quadratic penalty function, which satisfies the condition of limited growth due to the penalty parameter.

\section{Adaptive Method for Variational Inequalities with Relatively Bounded Operators}\label{adapt}

In this section we consider the Minty variational inequality problem \eqref{VI_problem} with relatively bounded \eqref{Rel_Boud} and monotone \eqref{monotone_operator} operator $g$, i.e.
\begin{equation}\label{Rel_Boud}
    \langle g(x),x-y\rangle \leqslant M\sqrt{2V(y,x)} \quad \forall x,y \in Q,
\end{equation}
for some $M>0$ and
\begin{equation}\label{monotone_operator}
    \langle g(y)-g(x),y-x\rangle \geqslant 0 \quad \forall x, y \in Q,
\end{equation}
where $Q$ is a convex compact set. In order to solve such a class of problems, we propose an adaptive algorithm, listed as Algorithm \ref{adaptive_alg3}, below. 
\begin{algorithm}[!ht]
\caption{Adaptive Algorithm for Variational Inequalities with Relatively Bounded Operators.}\label{adaptive_alg3}
\begin{algorithmic}[1]
   \REQUIRE $\varepsilon > 0, \rrev{x_{0} \in Q}, L_{0}>0$,  $R>0$ s.t.  $\max\limits_{x\in Q} V\left(x, x_{0}\right) \leqslant R^{2}, \rrev{k = 0.}$
   \STATE Set $ k = k+1, L_{k+1}=\frac{L_k}{2}$.
	\STATE Find
			\begin{equation*}\label{minVI}
			 x_{k+1} = \arg\min\limits_{x\in Q}\{\langle g(x_k),x \rangle + L_{k+1}V(x,x_{k})\}.
			\end{equation*}
	\STATE \textbf{if}
    \begin{equation*}\label{condVI}
        \frac{\varepsilon}{2} +  \langle g(x_k),x_{k+1}-x_k\rangle + L_{k+1} V(x_{k+1},x_k)\geqslant 0,
    \end{equation*}
    \textbf{then } go to the next iteration (item 1).
    \STATE \textbf{else}
    $$
        \text{set }L_{k+1}=2L_{k+1}, \text{ and go to item 2}.
    $$
    \STATE \textbf{end if}
	\STATE Stopping criterion
	\begin{equation*}
		S_N := \sum\limits_{k=0}^{N-1}\frac{1}{L_{k+1}} \geqslant \frac{2R^2}{\varepsilon}.
	\end{equation*}
	\ENSURE 
	$\widehat{x} = \frac{1}{S_N}\sum\limits_{k=0}^{N-1} \frac{x_{k+1}}{L_{k+1}}$.
\end{algorithmic}
\end{algorithm}

\begin{theorem}\label{the_VI}
Let  $g: Q\longrightarrow \mathbb {R}^n$ be a relatively bounded and monotone operator, i.e. \eqref{Rel_Boud} and \eqref{monotone_operator} hold, \rrev{$L_0\leqslant \frac{2M^2}{\varepsilon}$}. Then after the stopping of Algorithm \ref{adaptive_alg3}, the following inequality holds
$$
    \max\limits_{x\in Q}\langle g(x),\widehat{x}-x \rangle \leqslant \frac{1}{S_N}\sum\limits_{k=0}^{N-1}\frac{1}{L_{k+1}}\langle g(x),x_{k}-x \rangle \leqslant \varepsilon.
$$
Moreover, the total number of iterations will not exceed $N=\left\lceil\displaystyle\frac{4M^2R^2}{\varepsilon^2}\right\rceil.$
\end{theorem}
\begin{proof}
The proof is given in Appendix A.
\end{proof}

\rrev{
\begin{remark}\label{rem_gap}
Let us note, that defining 
\begin{equation*}
\Delta_N:= \frac{1}{S_N} \max _{x \in Q} \sum_{k=0}^{N-1} \frac{1}{L_{k+1}}\left\langle g\left(x_k\right), x_k-x\right\rangle,
\end{equation*}
one can get that the convergence for the function`s residuals 
\begin{equation*}
\min _{0 \leqslant k \leqslant N-1} f\left(x_k\right)-f^*,
\end{equation*}
for minimization problems with $g(x)$, defined as $g(x) = \nabla f(x)$, which also covers the primal-dual gap for saddle-point problems.
\end{remark}
}

Let us consider the following modification of Algorithm \ref{adaptive_alg3} with adaptation both to the parameters $L = \frac{M^2}{\varepsilon}$ and $\delta = \frac{\varepsilon}{2}$. 

\begin{algorithm}[!ht]
\caption{Adaptation to Inexactness for Variational Inequalities with Relatively Bounded Operators.}\label{Alg_50}
\begin{algorithmic}[1]
   \REQUIRE $\varepsilon > 0, \rrev{x_{0}\in Q}, L_{0}>0, \delta_0 >0$, $R$ s.t. $\max\limits_{x \in Q} V\left(x, x_{0}\right) \leqslant R^{2}, \rrev{k = 0}.$
    \STATE Set $k=k+1, L_{k+1}=\frac{L_k}{2}, \delta_{k+1}=\frac{\delta_k}{2}$.
    \STATE Find
	\begin{equation}\label{subproblem_alg50}
		x_{k+1} = \arg\min\limits_{x\in Q}\{\langle g(x_k),x \rangle + L_{k+1}V(x,x_{k})\}.
	\end{equation}
	%\IF{
	\STATE \textbf{if}
    	\begin{equation}\label{condalg50}
    	0 \leqslant \langle  g\left(x_{k}\right), x_{k+1}-x_{k}\rangle+L_{k+1} V(x_{k+1}, x_{k}) +\delta_{k+1},
    	\end{equation}
    \textbf{then} go to the next iteration (item 1).
   
    \STATE \textbf{else}
    $$
        \text{set }L_{k+1}=2\cdot L_{k+1}, \delta_{k+1}= 2\cdot\delta_{k+1}\text{ and go to item 2}.
    $$
 \STATE \textbf{end if}
\ENSURE $\widehat{x} = \frac{1}{S_N}\sum\limits_{k=0}^{N-1}\frac{x_{k+1}}{L_{k+1}}$.
\end{algorithmic}
\end{algorithm}

\begin{theorem}\label{theorem_estimate_alg50}
Let  $g: Q\longrightarrow \mathbb {R}^n$ be a relatively bounded and monotone operator, i.e. \eqref{Rel_Boud} and \eqref{monotone_operator} hold, \rrev{$L_0\leqslant 2L = \frac{2M^2}{\varepsilon}$}. Then after $N$ steps of Algorithm \ref{Alg_50} the following inequality holds
\begin{equation}\label{estimate_alg50}
    \max\limits_{x\in Q}\langle g(x),\widehat{x} - x \rangle  \leqslant \frac{R^2}{S_N} + \frac{1}{S_N} \sum_{k = 0}^{N-1} \frac{\delta_{k+1}}{L_{k+1}}.
\end{equation}
\rrev{Moreover, if $L_0\leqslant 2L$ and $\delta_0\leqslant \varepsilon$}, the auxiliary problem \eqref{subproblem_alg50} in Algorithm \ref{Alg_50} is solved no more than \rev{$2N + \log_2\frac{2L}{L_0}$ times.}
\end{theorem}
\begin{proof}
The proof is given in Appendix B.
\end{proof}

\rrev{
\begin{remark}
    Note, that Remark \eqref{rem_gap} also takes place for Theorem \eqref{theorem_estimate_alg50}.
\end{remark}}

\begin{remark}\label{remarkoptim}
The condition of the relative boundedness is essential only for justifying \eqref{condalg50}. For $L_{k+1} \geqslant L = \frac{M^2}{\varepsilon}$ and $\delta_{k+1} \geqslant \frac{\varepsilon}{2}$, \eqref{condalg50} certainly holds. 
So, \rrev{if $L_0\leqslant C_1 L$} for $C = \max\left\{ \rrev{C_1}; \, \frac{\varepsilon}{\delta_0}\right\}$, $L_{k+1} \leqslant C L$ and $\delta_{k+1} \leqslant \frac{C\varepsilon}{2}$ $\forall k \geqslant0$. Thus, $\max\limits_{x\in Q}\langle g(x),\widehat{x} - x \rangle  \leqslant \varepsilon$ after $N = O(\varepsilon^{-2})$ iterations of Algorithm \ref{Alg_50}. This fact, in essence, constitutes the optimality of the proposed method for the class of variational inequality problems with monotone $M$-relatively bounded operators.
%So, for $C = \max\left\{\frac{2L}{L_0}; \, \frac{2\delta}{\delta_0}\right\}$, $L_{k+1} \leqslant C L$ and $\delta_{k+1} \leqslant C\delta = \frac{C\varepsilon}{2}$ $\forall k \geqslant0$. Thus, $\max\limits_{x\in Q}\langle g(x),\widehat{x} - x \rangle  \leqslant \varepsilon$ after $N = O(\varepsilon^{-2})$ iterations of Algorithm \ref{Alg_50}. This fact, in essence, constitutes the optimality of the proposed method for the class of variational inequality problems with monotone $M$-relatively bounded operators.  
\end{remark}

%%%%%%%%%%%%%%%%%%%%%%%%%%%%%
\section{Adaptive Algorithms for Relatively Lipschitz Continuous Convex Optimization Problems}\label{adapt_univers}

Now we consider the classical optimization problem \eqref{main} under the assumption of $M$-relative Lipschitz continuity of the objective function $f$. For solving such a type of problems we propose two adaptive algorithms, listed as Algorithm \ref{adaptive_alg4} and Algorithm \ref{Alg_5}, below. 
\begin{algorithm}[!ht]
\caption{Adaptive Algorithm for Relatively Lipschitz Continuous Optimization Problems.}\label{adaptive_alg4}
\begin{algorithmic}[1]
   \REQUIRE $\varepsilon > 0, \rrev{x_{0}\in Q}, L_{0}>0$, $R$ s.t. $V\left(x_{*}, x_{0}\right) \leqslant R^{2}, \rrev{k = 0}.$
   \STATE Set $ k = k+1, L_{k+1}=\frac{L_k}{2}$.
	\STATE Find
		\begin{equation*}\label{subproblem_alg4}
			 x_{k+1} = \arg\min\limits_{x\in Q}\{\langle\nabla f(x_k),x \rangle + L_{k+1}V(x,x_{k})\}.
		\end{equation*}
	\STATE \textbf{if}
    \begin{equation}\label{condalg4}
    	0 \leqslant \langle  \nabla f\left(x_{k}\right), x_{k+1}-x_{k}\rangle+L_{k+1} V(x_{k+1}, x_{k}) +\frac{\varepsilon}{2},
    \end{equation}
    \textbf{then} go to the next iteration (item 1). 
    \STATE \textbf{else}
    $$
        \text{set }L_{k+1}=2\cdot L_{k+1} \text{ and go to item  2}.
    $$
    \STATE \textbf{end if}
    \STATE \rev{Stopping criterion
	\begin{equation*}
		S_N = \sum\limits_{k=0}^{N-1}\frac{1}{L_{k+1}} \geqslant \frac{2R^2}{\varepsilon}.
	\end{equation*}}
	\ENSURE $\widehat{x} = \frac{1}{S_N}\sum\limits_{k=0}^{N-1} \frac{x_{k+1}}{L_{k+1}}$.
\end{algorithmic}
\end{algorithm}

\begin{theorem}\label{theorem_adaptive_Alg_2}
Let  $f: Q\longrightarrow \mathbb {R}$ be a convex and $M$-relatively Lipschitz continuous function, i.e. \eqref{eqalpha1relsm} and \eqref{eqalpha1relsm1} take place with $\alpha = 1, \delta = \frac{\varepsilon}{2}, L \rev{\geq} \frac{M^2}{\varepsilon}$. Then after \alex{the stopping of} Algorithm \ref{adaptive_alg4}, the following inequality holds $ f(\widehat{x})-f(x_*)\leqslant \varepsilon$. 
Moreover, the total number of iterations will not exceed $N=\left\lceil\displaystyle\frac{4M^2R^2}{\varepsilon^2}\right\rceil.$
\end{theorem}

\begin{proof}
The proof is given in Appendix C.
\end{proof}

\begin{algorithm}[!ht]
\caption{Adaptation to Inexactness for Relatively Lipschitz Continuous Optimization Problems.}\label{Alg_5}
\begin{algorithmic}[1]
   \REQUIRE $\varepsilon > 0, \rrev{x_{0}\in Q}, L_{0}>0, \delta_0 >0$, $R$ s.t. $V\left(x_{*}, x_{0}\right) \leqslant R^{2}, \rrev{k = 0}.$
    \STATE Set $k=k+1, L_{k+1}=\frac{L_k}{2}, \delta_{k+1}=\frac{\delta_k}{2}$.
    \STATE Find
   \begin{equation}\label{subproblem_alg5}
		x_{k+1} = \arg\min\limits_{x\in Q}\{\langle \nabla f(x_k),x \rangle + L_{k+1}V(x,x_{k})\}.
	\end{equation}
	\STATE \textbf{if}
    \begin{equation*}\label{condalg5}
    0 \leqslant \langle  \nabla f\left(x_{k}\right), x_{k+1}-x_{k}\rangle+L_{k+1} V(x_{k+1}, x_{k}) +\delta_{k+1},
    \end{equation*}
    \textbf{then} go to the next iteration (item 1).
    \STATE \textbf{else}
    $$
        \text{set }L_{k+1}=2\cdot L_{k+1}, \delta_{k+1}= 2\cdot\delta_{k+1}\text{ and go to item 2}.
    $$
    \STATE \textbf{end if}
	\ENSURE $\widehat{x} = \frac{1}{S_N}\sum\limits_{k=0}^{N-1}\frac{x_{k+1}}{L_{k+1}}$.
\end{algorithmic}
\end{algorithm}

\begin{theorem}\label{theorem_estimate_alg5}
Let  $f: Q\longrightarrow \mathbb {R}$ be a convex and $M$-relatively Lipschitz continuous function, i.e. \eqref{eqalpha1relsm} and \eqref{eqalpha1relsm1} take place with $\alpha = 1, \delta = \frac{\varepsilon}{2}, L = \frac{M^2}{\varepsilon}$. Then after \alex{$N$ steps of} Algorithm \ref{Alg_5}, the following inequality holds
\begin{equation}\label{estimate_alg5}
    f(\widehat{x})-f(x_*) \leqslant \frac{R^2}{S_N} + \frac{1}{S_N} \sum_{k = 0}^{N-1} \frac{\delta_{k+1}}{L_{k+1}}.
\end{equation}
\end{theorem}

\begin{proof}
The proof is similar to the proof of Theorem \ref{theorem_adaptive_Alg_2} with $\displaystyle\frac{\varepsilon}{2} \longrightarrow \delta_{k+1}.$
\end{proof}

The optimality of Algorithm \ref{Alg_5} for the class of convex and $M$-relatively Lipschitz continuous problems can be proved similar to Remark \ref{remarkoptim}.

%%%%%%%%%%%%%%%%%%%%%%%%%%%%%%%%%%%%%%%%%%%%%%%%%%%
%Section Universal Algorithms %%%%%%%%%%%%%%%%%%%%%%%%%%%%%%%%%%%%%%%%%%%

\section{Universal Algorithms for Relatively Smooth and Relatively Lipschitz Continuous Convex Optimization Problems}\label{sect_univers}

In this section, we introduce some analogues of Algorithms \ref{adaptive_alg4} and \ref{Alg_5}, which adjust to the "degree of relative smoothness" of the considered $(\alpha, L, \delta)$-relatively smooth problem. This approach allows the construction of adaptive gradient-type methods that are applicable to both relatively Lipschitz continuous and relatively smooth problems with optimal complexity estimates.

\begin{algorithm}[!ht]
\caption{Universal Method for Relatively Smooth and Relatively Lipschitz Continuous Convex Optimization Problems with Adaptation to Inexactness.}\label{Algor2}
\begin{algorithmic}[1]
   \REQUIRE $\varepsilon > 0, \rrev{x_{0} \in Q}, L_{0}>0,\delta_0>0$, $R$ s.t. $V\left(x_{*}, x_{0}\right) \leqslant R^{2}, \rrev{k=0}.$
   \STATE Set  $k = k+1, L_{k+1}=\frac{L_k}{2}, \delta_{k+1}=\frac{\delta_k}{2}$.
	\STATE Find
	\begin{equation}\label{eqproblem}
	x_{k+1} = \arg\min\limits_{x\in Q}\{\langle \nabla f(x_k),x \rangle + L_{k+1}V(x,x_{k})\}.
	\end{equation}
	\STATE \textbf{If}
    \begin{equation*}\label{condalg2}
    f\left(x_{k+1}\right) \leqslant f\left(x_{k}\right)+\langle  \nabla f\left(x_{k}\right), x_{k+1}-x_{k}\rangle+L_{k+1} (V(x_{k+1}, x_{k})+ V(x_{k}, x_{k+1}))+\delta_{k+1},
    \end{equation*}
    \textbf{then} go to the next iteration (item 1).
    \STATE \textbf{else} 
    $$\text{set }L_{k+1}=2\cdot L_{k+1}, \delta_{k+1}=2\cdot\delta_{k+1} \text{ and go to item 2}.$$
    \STATE \textbf{end if}  
	\ENSURE $\widehat{x} = \frac{1}{S_N}\sum\limits_{k=0}^{N-1}\frac{x_{k+1}}{L_{k+1}}.$  
\end{algorithmic}
\end{algorithm}

\begin{theorem}\label{theorem_Algor2}
Let  $f: Q\longrightarrow \mathbb {R}$ be a convex and $(\alpha, L, \delta)$-relatively smooth function, i.e. \eqref{eqalpha1relsm}, \eqref{eqalpha1relsm1} hold. Then after $N$ iterations of Algorithm \ref{Algor2}, the following inequality holds
\begin{equation*}\label{equat2}
    f(\widehat{x})-f(x_{*})\leqslant \frac{R^2}{S_N}+ \frac{1}{S_N}\sum\limits_{k=0}^{N-1} \frac{\delta_{k+1}}{L_{k+1}},
\end{equation*}
where $S_N = \sum\limits_{k=0}^{N-1}\frac{1}{L_{k+1}}.$
Note that for \color{black}{$L_0 \leqslant 2L$ and $\delta_0 \leqslant 2\delta$} the auxiliary problem \eqref{eqproblem} in Algorithm \ref{Algor2} is solved no more than $2N + \log_2\frac{2L}{L_0}$ times.
\end{theorem}

\begin{proof}
The proof is given in Appendix D.
\end{proof}

The optimality of Algorithm \ref{Algor2} for the class of convex and $M$-relatively Lipschitz continuous problems can be proved similar to Remark \ref{remarkoptim}. The optimal rate of convergence $O(\varepsilon^{-1})$ for the class of $L$-relatively smooth problems also takes place for Algorithm \ref{Algor2}. \alex{For more details see the conclusion of proof in Appendix E, the proof of these facts for Algorithm \ref{Algor2} can be obtained analogously.}

Let us now formulate a variant of the universal method for relatively Lipschitz continuous and relatively smooth problems which makes it possible to prove the guaranteed preservation of the optimal complexity estimates. This method is listed as Algorithm \ref{universal_alg2}, below.
%for the rate of convergence $O\left(\varepsilon^{-2}\right)$

\begin{algorithm}[!ht]
\caption{Universal Method for Relatively Smooth and Relatively Lipschitz Continuous Convex Optimization Problems.}\label{universal_alg2}
\begin{algorithmic}[1]
   \REQUIRE $\varepsilon > 0, \rrev{x_{0}\in Q}, L_{0}>0$, $R$ s.t. $V\left(x_{*}, x_{0}\right) \leqslant R^{2}, \rrev{k = 0}.$
   \STATE Set $k =k+1, L_{k+1}=\frac{L_k}{2}$.
	\STATE Find
\begin{equation*}
	x_{k+1} = \arg\min\limits_{x\in Q}\{\langle \nabla f(x_k),x \rangle + L_{k+1}V(x,x_{k})\}.
	\end{equation*}
	%\IF{
	\STATE \textbf{If}
    \begin{equation*}
    f\left(x_{k+1}\right) \leqslant f\left(x_{k}\right)+\langle  \nabla f\left(x_{k}\right), x_{k+1}-x_{k}\rangle + L_{k+1} ( V(x_{k+1}, x_{k})  V(x_{k}, x_{k+1})) +\frac{3 \varepsilon}{4},
    \end{equation*}
     \textbf{then} go to the next iteration (item 1).
    \STATE \textbf{else}
    $$\text{set }L_{k+1}=2\cdot L_{k+1} \text{ and go to item 2.}$$
    \STATE \textbf{end if}
	\STATE Stopping criterion
	\begin{equation*}
		S_N := \sum\limits_{k=0}^{N-1}\frac{1}{L_{k+1}}\geqslant \frac{4R^2}{\varepsilon}.
	\end{equation*}
	\ENSURE $\widehat{x} = \frac{1}{S_N}\sum\limits_{k=0}^{N-1}\frac{x_{k+1}}{L_{k+1}}$.
\end{algorithmic}
\end{algorithm}

%\rrev{$L_0\leq\frac{2M^2}{\varepsilon}$}

\begin{theorem}\label{ThmUnivMeth2}
\rev{Let  $f: Q\longrightarrow \mathbb {R}$ be a convex and $(\alpha, L, \delta)$-relatively smooth function, i.e. \eqref{eqalpha1relsm} and \eqref{eqalpha1relsm1} hold with $\delta \leqslant \frac{3\varepsilon}{4}$, \rrev{$L_0\leqslant 2L$}. Then after the stopping of Algorithm  \ref{universal_alg2}, the following inequality holds $f(\widehat{x})-f(x_*) \leqslant\varepsilon.$ The number of iterations of Algorithm \ref{universal_alg2} does not exceed $\displaystyle \left\lceil\frac{8LR^2}{\varepsilon}\right\rceil$. If $f$ is $\left(1, \frac{2M^2}{\varepsilon}, \frac{\varepsilon}{2}\right)$-relatively smooth function (for example, $M$-relatively Lipschitz continuous function) then the number of iterations of Algorithm \ref{universal_alg2} does not exceed $\left\lceil\displaystyle\frac{16M^2R^2}{\varepsilon^2}\right\rceil$.}
\end{theorem}

\begin{proof}
The proof is given in Appendix E.
\end{proof}
\textcolor{black}{
\begin{remark}
    It is worth noting that, generally speaking, for Algorithms \ref{adaptive_alg4}--\ref{universal_alg2} it is acceptable to use the following output point
    \begin{equation*}
        \widehat{x} = \arg\min\limits_{i\in\{0,\ldots,N+1\}} f(x_i).
    \end{equation*}
At the same time for various applied problems, such a modification can both improve and degrade the practical quality of the algorithms. 
\end{remark}}

%%%%%%%%%%%%%%%%%%Section  Numerical Experiments  %%%%%%%%%%%%%%%%%%%%%%%%%%%%%%%%%%%%%%%%%%%%%%%%%%
\section{Numerical Experiments}\label{experiments_Alg5}
In this section, in order to demonstrate the performance of the proposed Algorithms,  we  firstly consider some numerical experiments concerning the Intersection of Ellipsoids Problem (IEP). Secondly, we compare  the proposed Algorithm \ref{Alg_5} with AdaMirr algorithm, which was recently proposed in \cite{AdaMirr_2021}.  We also, consider some numerical experiments concerning the Support Vector Machine (SVM) \cite{Lu,pegasos_2011}.

%we  firstly  compare  the proposed Algorithm \ref{Alg_5} with AdaMirr algorithm, which was recently proposed in \cite{AdaMirr_2021}. We also, consider some numerical experiments concerning the Intersection of Ellipsoids Problem (IEP) and  Support Vector Machine (SVM) \cite{Lu,pegasos_2011}. 

All experiments were implemented in Python 3.4, on a computer fitted with Intel(R) Core(TM) i7-8550U CPU @ 1.80GHz, 4 Core(s), 8 Logical Processor(s). The RAM of the computer is 8 GB.

\subsection{The Intersection of Ellipsoids Problem (IEP)}\label{IEP}
For the Intersection of Ellipsoids Problem, \alex{supposing, that the intersection is nonempty}, we compute a point $x \in \mathbb{R}^n$ in the intersection of $m$ ellipsoids, i.e.

%$x \in Q \subseteq \mathbb{R}^n$ \moh{($Q$ is a closed convex set)} in the intersection of $m$ ellipsoids, i.e.

\begin{equation*}
    x \in \mathcal{E} = \mathcal{E}_1 \cap \mathcal{E}_2 \cap \ldots \cap \mathcal{E}_m, 
\end{equation*}
where $\mathcal{E}_{i}=\left\{ \rrev{x \in \mathbb{R}^n}: \frac{1}{2} x^{T} A_{i} x+ \rrev{b_{i}^T x}+c_{i} \leqslant 0\right\}$, $A_i \in \mathbb{R}^{n\times n}$ is a given symmetric positive semi-definite matrix, $b_i \in \mathbb{R}^n, c_i \in \mathbb{R}$ are given, for every $i =1, \ldots, m$. We note that the Intersection of Ellipsoids Problem is equivalent to the following unconstrained optimization problem 
\begin{equation}\label{objective_IEP}
    \rev{\min\limits_{x\in \mathbb{R}^n} \left\{ f(x) := \max\limits_{\rrev{1\leqslant i\leqslant m}} \left[
    \frac{1}{2} x^{T} A_{i} x + \rrev{b_{i}^T x}+c_{i} \right] \right\}.}
\end{equation}
The objective function $f$ in \eqref{objective_IEP} is both non-differentiable and non-Lipschitz  %\alex{if $Q$ is not compact. Even for convex compact set $Q$ the values of the Lipschitz constants of such a function can be extremely large.}
\moh{\cite{Lu}}. So the traditional first-order methods are not applicable to such types of problems. We will demonstrate, how the proposed Algorithm \ref{Alg_5} can be applied to solve such a problem (here we take more attention to the Algorithm \ref{Alg_5}, because it works better than Algorithms \ref{adaptive_alg4} and \ref{Algor2}, see Fig. \ref{results_alg345_IEP}).

%(note that we can use the other proposed algorithms, but in practice, we found that the Algorithm \ref{Alg_5} is the best). 

Let $\sigma:=\max\limits_{1 \leqslant i \leqslant m}\left\|A_{i}\right\|_{2}^{2}$ where $\left\|A_{i}\right\|_{2}$ is the \rrev{spectral norm} of $A_i$,\\
$\rho :=\max\limits_{1 \leqslant i \leqslant m}\left\|A_i b_i \right\|_{2}$ and $\gamma:=\max\limits_{1 \leqslant i \leqslant m}\left\|b_i \right\|_{2}^2$.  We run Algorithm \ref{Alg_5} with the following prox function 
\begin{equation}\label{prox_function_IEP}
d(x):=\frac{a_2}{4}\|x\|_{2}^{4}+\frac{a_1}{3}\|x\|_{2}^{3}+\frac{a_0}{2}\|x\|_{2}^{2},
\end{equation}
where $a_0 = \gamma, a_1 = \rho, a_2 = \sigma$ (see \cite{Lu} for more details). The objective function $f$ \eqref{objective_IEP} is $1$-relatively Lipschitz continuous with respect to the prox function $d(\cdot)$, defined in \eqref{prox_function_IEP} \rrev{\cite{Lu}}. The Bregman divergence $V(\cdot, \cdot)$ for the corresponding prox function $d(\cdot)$ is defined as follows
\begin{equation}\label{Bregmann_IEP}
    V(y,x) = a_0 V_{d_0}(y,x) + a_1 V_{d_1}(y,x) + a_2 V_{d_2}(y,x).  
\end{equation}
\rrev{where $d_i(x) = \frac{1}{i+2}\|x\|_2^{i+2} \; (i=0, 1, 2)$,} and 
$$
    V_{d_i}(y,x) = \frac{1}{i+2} \left(\|y\|_{2}^{i+2}+(i+1) \|x\|_{2}^{i+2}-(i+2)\|x\|_{2}^{i}\langle x, y\rangle\right)\; (i = 0, 1, 2).
$$
Note that each iteration of  Algorithm \ref{Alg_5} requires the capability to solve the subproblem \eqref{subproblem_alg5}, which is equivalent to the following linearized problem
\begin{equation}\label{equivalent_subproblem_alg5}
	x_{k+1} = \arg\min\limits_{x\in \mathbb{R}^n }\{\langle c_k,x \rangle + d(x)\},
\end{equation}
where $c_k = \frac{1}{L_{k+1}} \nabla f(x_k) -\nabla d(x_k)$ and $d(x)$ is given in \eqref{prox_function_IEP}. The solution of the problem \eqref{equivalent_subproblem_alg5} can be found explicitly
\begin{equation*}
    x_{k+1} = -\theta_k c_k,
\end{equation*}
for some $\theta_k \geqslant 0$, where $\theta_k$ is a positive real root of the following cubic equation %(this root is calculated by the  np.root() function, which is a ready function in NumPy to calculate the roots of a polynomial with given coefficients)
\begin{equation*}\label{cubic_eq_IEP}
    \gamma \theta + \rho\|c_k\|_2 \theta^2 + \sigma \|c_k\|_2^2 \theta^3 -1 = 0.
\end{equation*}

We run Algorithm \ref{Alg_5} with different values of $n$  and prox-function \eqref{prox_function_IEP} and the starting point $x_0 = \left(0.2, \ldots, 0.2\right) \in \mathbb{R}^n$ and not in $\mathcal{E}$.  \rrev{The matrices $A_i$, for every $i = 1, \ldots, m$,  are diagonal matrices with entries  chosen randomly from the uniform distribution over $(0, 1)$}, the vectors $b_i$ and the constants $c_i$, are also chosen randomly from a normal (Gaussian) distribution with mean (center) equaling $0$ and standard deviation (width) equaling $0.1$. \rrev{ We generated the random data 5 times and averaged the results of algorithms that we received each time, such that the $\mathbf{0} \in \mathbb{R}^n \in \mathcal{E}$.} We considered $L_0 = \frac{\|\nabla f(1,0,\ldots,0) - \nabla f(0,1,0, \ldots, 0)\|_2}{\sqrt{2}}, \delta_0 = 0.5$, and $R^2 = \frac{3\sigma}{4}\|x_0\|_2^3 + \frac{2\rho}{3}\|x_0\|_2^3 + \frac{\gamma}{2}\|x_0\|_2^2$ (see the proof of the Proposition 5.4 in \cite{Lu}).

The results of the work of Algorithm \ref{Alg_5} for IEP are presented in Fig \ref{results_alg4_IEP}, below.
These results demonstrate the running time of the algorithm in seconds as a function of the number of iterations, and the quality of the solution "Estimate", which is in fact the right side of inequality \eqref{estimate_alg5}.

We note that the quality of the solution of the problem, which is produced by Algorithm \ref{Alg_5}, grows sharply at the beginning of the work of the algorithm for $k=100$ to $10\,000$. \rrev{We improve the quality of the initial solution by two orders of magnitude on average.} Nevertheless, %the growth of the quality of the solution slows down significantly with the growth of the number of iterations from
\rrev{the rate of convergence significantly decreases when the number of iterations goes from $k=15\,000$ to $100\,000$.}

\begin{figure}[htp]
	\minipage{0.50\textwidth}
	\includegraphics[width=\linewidth]{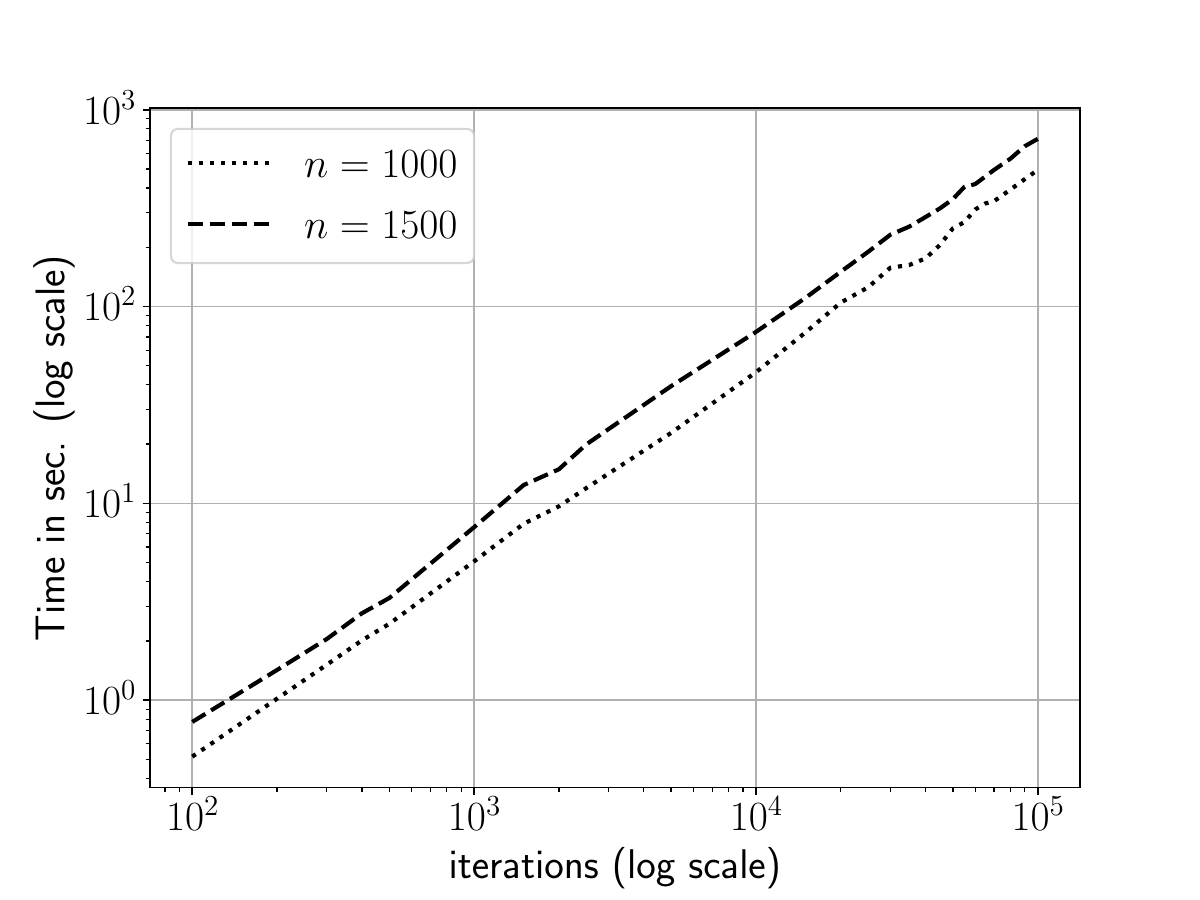}
	\endminipage\hfill
	\minipage{0.50\textwidth}
	\includegraphics[width=\linewidth]{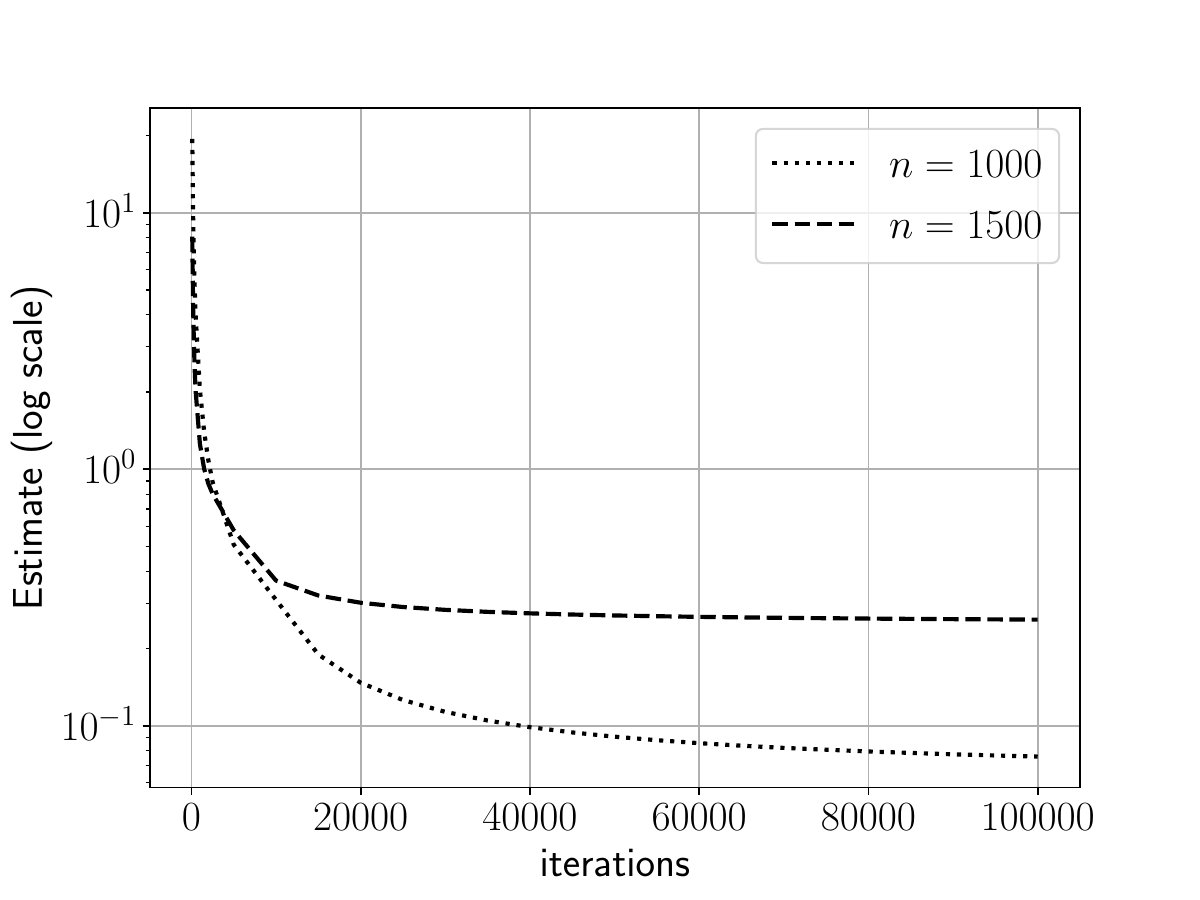}
	\endminipage\hfill
	\caption{The results of Algorithm \ref{Alg_5} for IEP with different values of $n$  and $m = 10$.}
	\label{results_alg4_IEP}
\end{figure}

\begin{figure}[htp]
	\minipage{0.50\textwidth}
	\includegraphics[width=\linewidth]{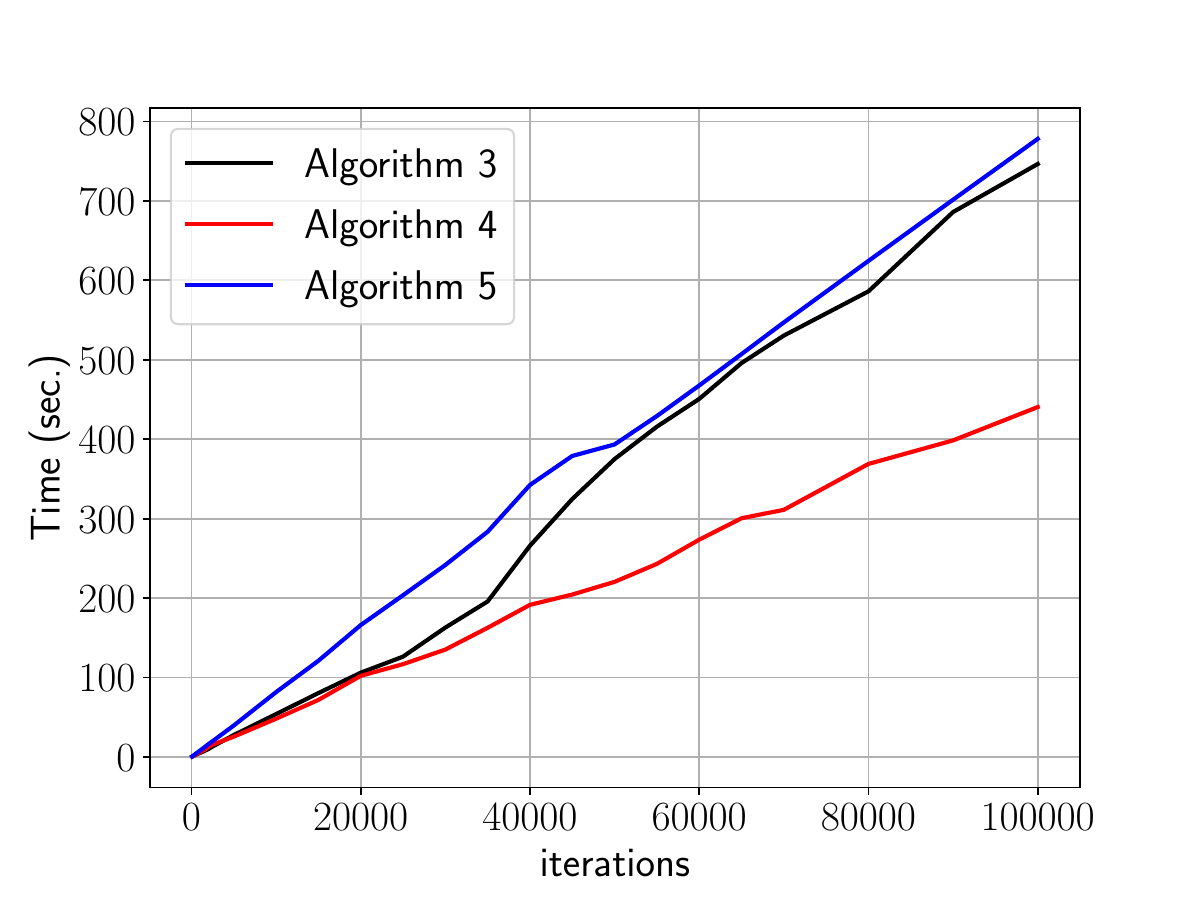}
	\endminipage\hfill
	\minipage{0.50\textwidth}
	\includegraphics[width=\linewidth]{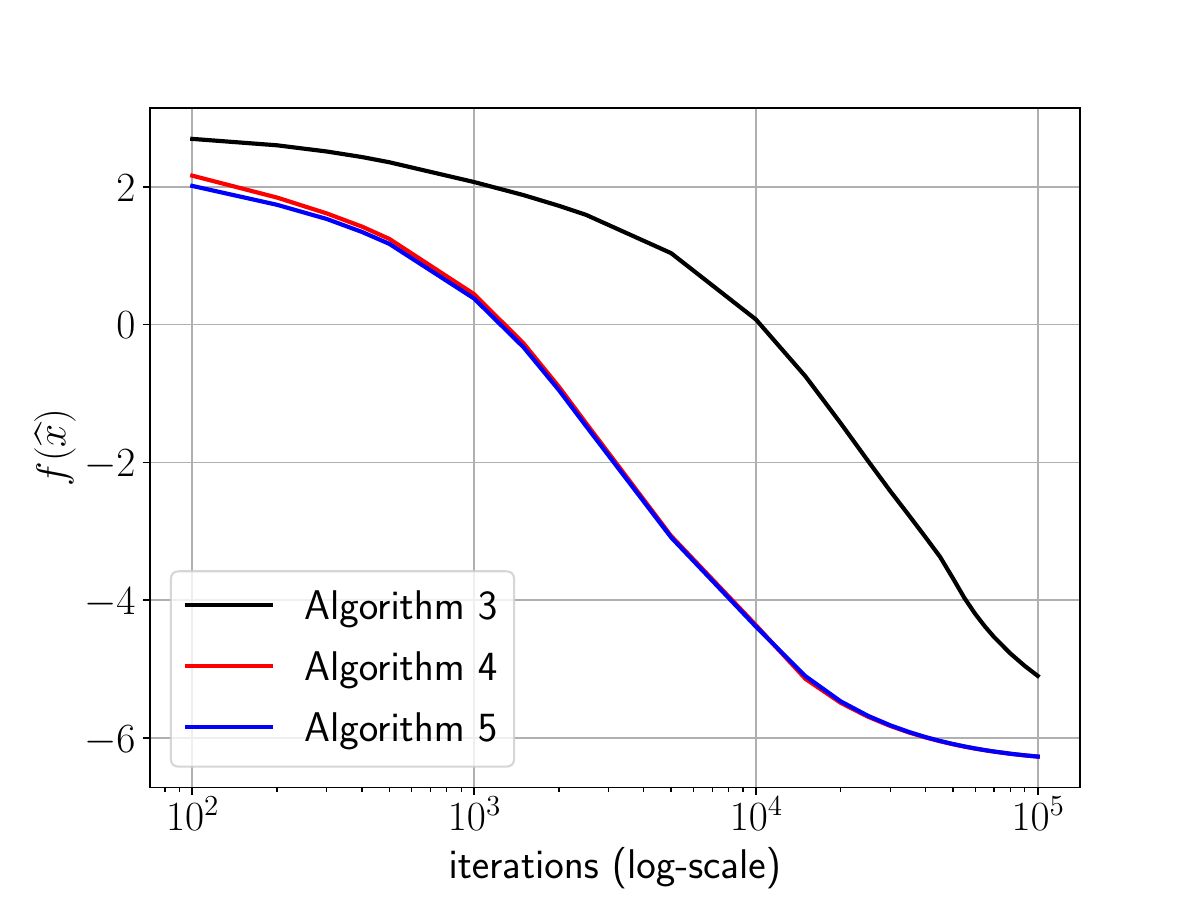}
	\endminipage\hfill
	\caption{The results of Algorithms \ref{adaptive_alg4}, \ref{Algor2} and \ref{Alg_5} for IEP with $n=1000$  and $m = 10$.}
	\label{results_alg345_IEP}
\end{figure}

%\begin{table}[!ht]
%	\centering
%	\caption{The results of Algorithm \ref{Alg_5} for IEP with different values of $n$  and $m = 10$.}
%	\label{table_alg5_2}
%	\begin{tabular}{|c|c|c|c|c|}
%		\hline
%		\multirow{2}{*}{} & \multicolumn{2}{c|}{$n =1000$} & \multicolumn{2}{c|}{$n = 1500$} \\ \cline{2-5} 
%		$k$	 & Time (sec.) &Estimate & Time (sec.) & Estimate\\ \hline
%		100	 &1.466 &221.4589 &1.751 &68.4040 \\ 
%		200	 &2.265 &111.4047 &3.375 &21.8705  \\ 
%		300	 &3.452 &74.4762 &4.936 &11.4712 \\ 
%		400  &4.900 &56.3878 &6.581 & 7.8800 \\ 
%		500	 & 5.803 &45.1152 &8.915 &6.0600 \\ 
%		1000 &12.863 &17.1154 &16.276 &2.9899 \\ 
%		2500 &30.610 &5.3165 &42.192 & 1.4373\\ 
%		5000 &59.891 &2.5314 &78.378 & 0.9596\\
%		10000&125.569 &1.2858 &174.444 &0.7276 \\
%		15000&179.651 &0.8899 &277.783 &0.6513 \\
%		20000&234.114 &0.6954 &347.313 & 0.6133\\ 
%		25000&304.119 &0.5797 &445.681 & 0.5905\\
%		30000&357.322 &0.5031 &515.057 & 0.5754\\
%		35000&406.205 &0.4485 &617.556 & 0.5646\\
%		40000&469.334 &0.4077 &676.006 &0.5565 \\
%		45000&536.950 &0.3761 &701.034 & 0.5502\\  
%		50000&583.905 &0.3508 &780.465 & 0.5452 \\ 
%		55000&650.166 &0.3302 &825.866 & 0.5411\\
%		60000&660.542 &0.3130 &832.798 & 0.5376\\
%		65000&708.685 &0.2984 &888.845 &0.5347 \\
%		70000&711.563 &0.2860 &895.642 &0.5323 \\
%		80000&721.268 &0.2658 &1007.684 &0.5281 \\ 
%		90000&740.335 &0.2501 &1024.745 &0.5251 \\  
%		100000&829.853 &0.2375 &1087.373 &0.5226 \\\hline  
%	\end{tabular}
%\end{table}

%%%%%%%%%%%%%%%%%%%%%%%%%%%%%%%%%%%%%%%%
\subsection{Comparison with AdaMirr} 
Recently in \cite{AdaMirr_2021}, \rrev{there was proposed} an adaptive first-order method, called AdaMirr, in order to solve the relatively continuous and relatively smooth optimization problems. \rrev{AdaMirr briefly can be stated as
\begin{equation*}
    x_{k+1} = \arg\min\limits_{x\in Q}\{\langle - \gamma_k \nabla f(x_k),x_k - x \rangle + V(x,x_k)\}, \quad k = 1, 2, \ldots.
\end{equation*}
with $\gamma_k$ defined as
$$
\gamma_k = \frac{1}{\sqrt{\sum_{s=0}^{k-1} \delta_s^2}} \quad \text {with} \quad  \delta_s^2=\frac{V\left(x_s, x_{s+1}\right)+V\left(x_{s+1}, x_s\right)}{\gamma_s^2}, \quad k = 1, 2, \ldots
$$
and $\delta_0 = \sqrt{V(x_0,x_1) + V(x_1,x_0)}$. 
In \cite{AdaMirr_2021}, it was proved that for $M$-relatively Lipschitz continuous convex function, then after $N$ steps of AdaMirr, the following inequality holds
\begin{equation}\label{estim_adamirr}
f\left(\overline{x}_N\right)-f^* \leqslant \frac{\sqrt{2} M\left[D_1+ \frac{8 M^2}{\delta_0^2} +2 \ln \left(1 + \frac{2 M^2 N}{\delta_0^2} \right)\right]}{\sqrt{N}}+\frac{3 \sqrt{2} M + \frac{4 M^2}{\delta_0^2}}{N},
\end{equation}
where $\overline{x}_N = \frac{1}{N} \sum_{k=1}^{N} x_k$ and $D_1 = V(x_*, x_1)$.
}

\rev{In this subsection,} we compare the proposed Algorithm \ref{Alg_5} with AdaMirr, for the Intersection of Ellipsoids Problem  (see Subsec. \ref{IEP}). We run the compared algorithms for the same parameters and setting \rrev{as in the Subsec.} \ref{IEP}. The results of the comparison are presented in Fig. \ref{fir_n1000}, which illustrates  the value of the objective function at the output point of each algorithm, the estimates of the quality of the solution for Algorithm \ref{Alg_5} (see the right side of inequality \eqref{estimate_alg5}) and AdaMirr (see the right side of inequality \eqref{estim_adamirr}) and the running time of algorithms in seconds.

From the results in Fig. \ref{fir_n1000}, we can see that the proposed Algorithm \ref{Alg_5} works better than AdaMirr, except for the running time, where AdaMirr works faster because, in the proposed Algorithm \ref{Alg_5}, there is an adaptive procedure for the Lipschitz continuity parameter, which needs more time. Note that AdaMirr does not converge to the solution of the problem for all taken iterations from $100$ to $10^4$. 
\begin{figure}[htp]
	\minipage{0.32\textwidth}
	\includegraphics[width=\linewidth]{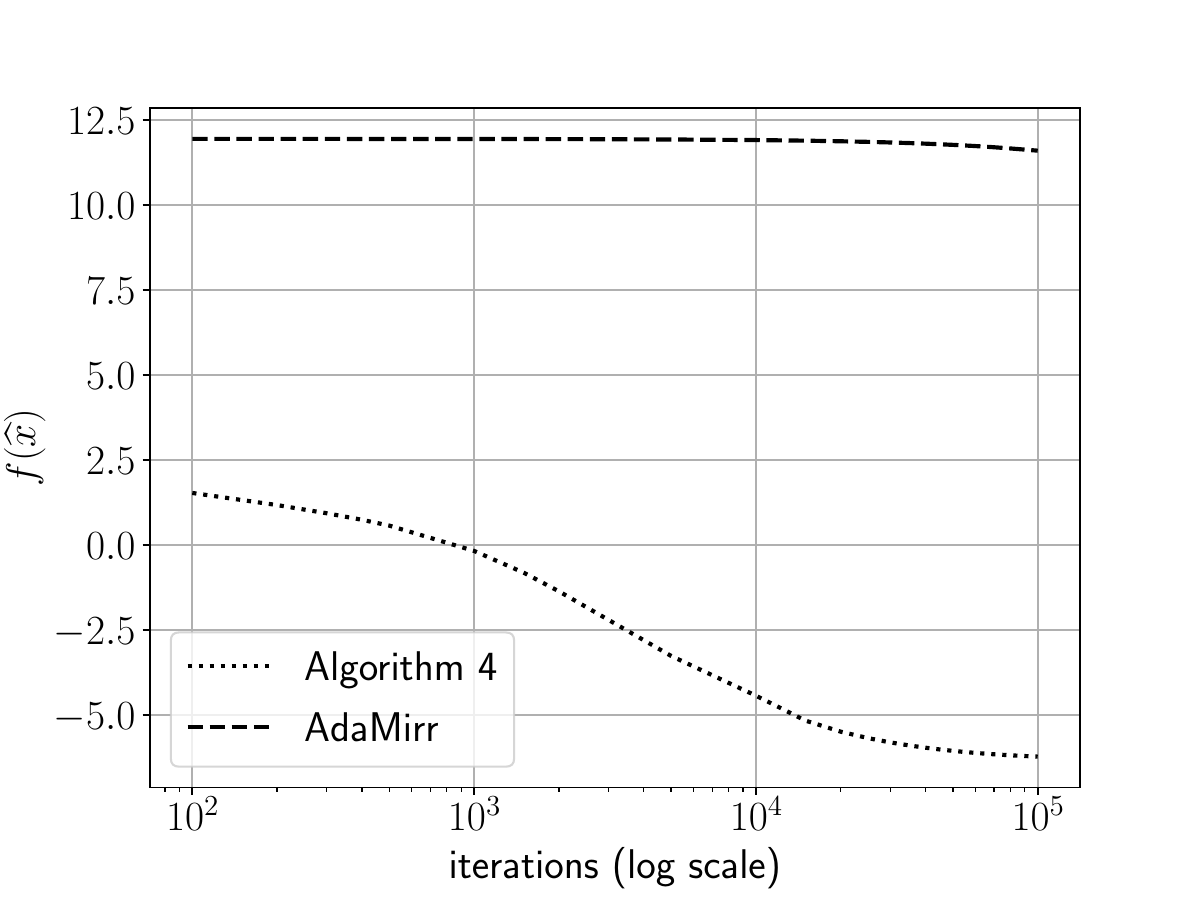}
	\endminipage\hfill
	\minipage{0.32\textwidth}
	\includegraphics[width=\linewidth]{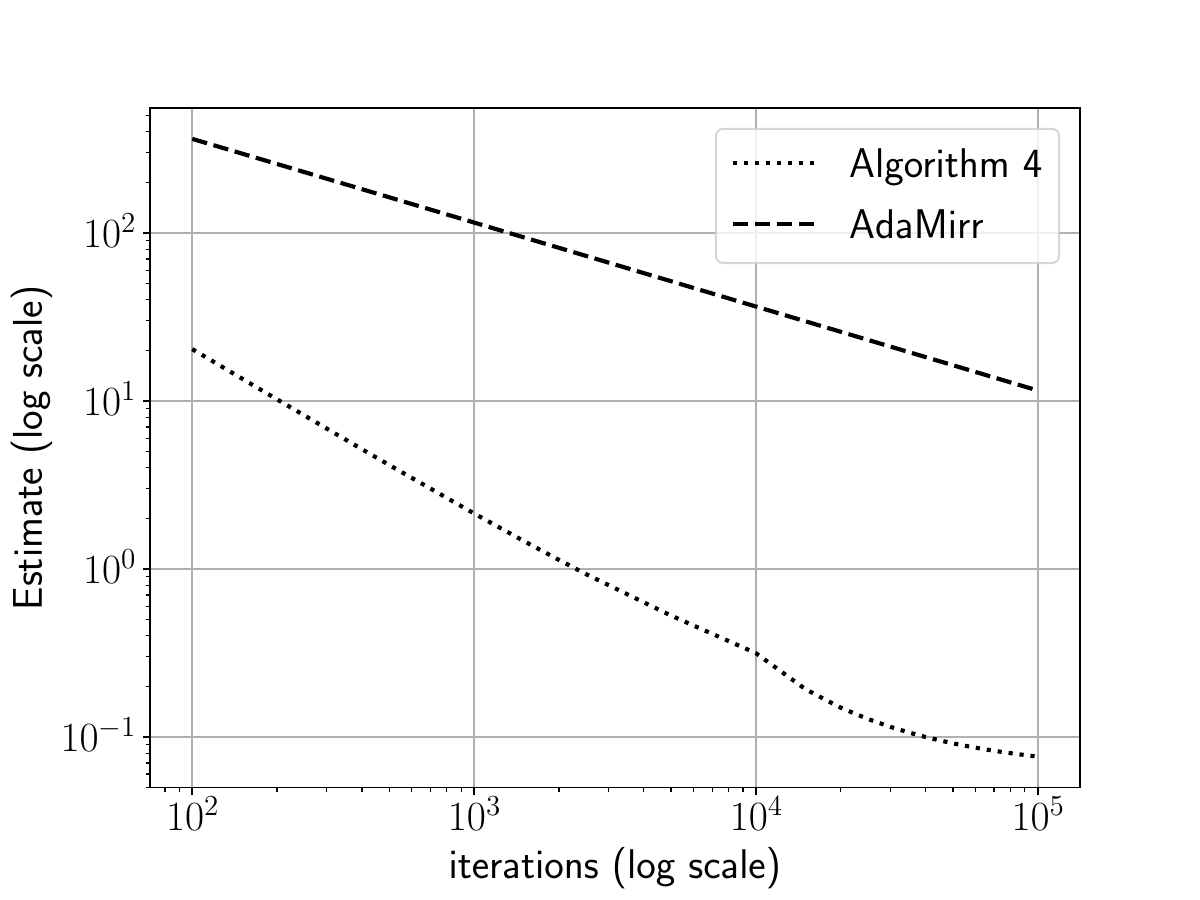}
	\endminipage\hfill
	\minipage{0.32\textwidth}%
	\includegraphics[width=\linewidth]{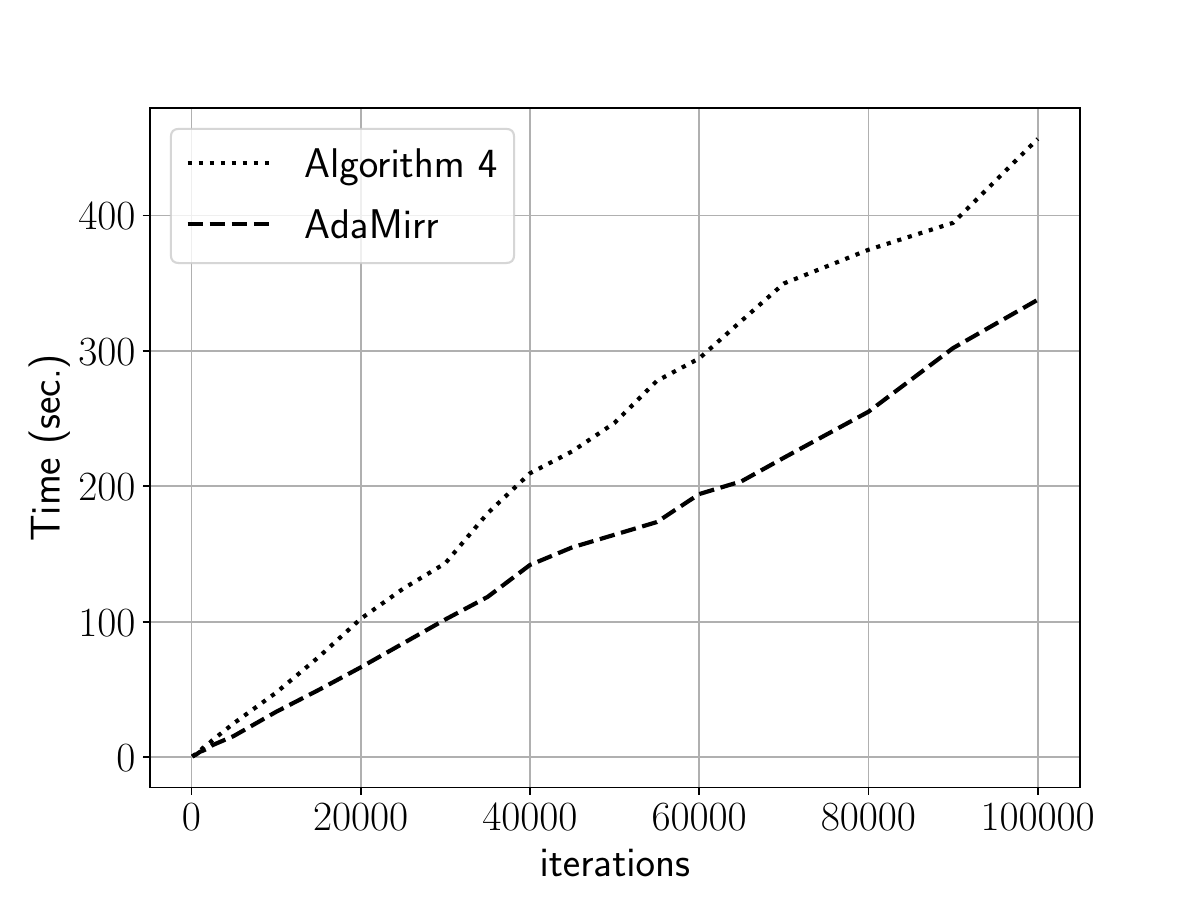}
	\endminipage
	\caption{The results of comparison of Algorithm \ref{Alg_5} and AdaMirr for IEP with $n = 1000$ and $m=10$.}
	\label{fir_n1000}
\end{figure}

%\begin{figure}[htp]
	%\minipage{0.32\textwidth}
	%\includegraphics[width=\linewidth]{Figs/value_n1500}
	%\endminipage\hfill
	%\minipage{0.32\textwidth}
	%\includegraphics[width=\linewidth]{Figs/estimate_n1500}
	%\endminipage\hfill
	%\minipage{0.32\textwidth}%
	%\includegraphics[width=\linewidth]{Figs/time_n1500}
	%\endminipage
	%\caption{The results of comparison of Algorithm \ref{Alg_5} and AdaMirr for IEP with $n = 1500$ and $m=10$.}
	%\label{fir_n1500}
%\end{figure}

%%%%%%%%%%%%%%%%%%%%%%%%%%%%%%%%%%%%%%%%%
\subsection{Support Vector Machine (SVM) and Inequality-Type Function Constraints}
The Support Vector Machine (SVM) is an important supervised
learning model for binary classification problem \cite{pegasos_2011}. 
The SVM optimization problem can be formulated as follows
\begin{equation}\label{objective_SVM}
\rev{\min\limits_{x\in \widetilde{Q} } \left\{ f(x) := \left(\frac{1}{n}\sum_{i=1}^{n}\max\{0, 1- y_i x^\top w_i\}\right) +\frac{\moh{\tau}}{2}\|x\|_2^2\right\}, }
\end{equation}
where $w_i$ is the input feature vector of sample $i$ and $y_i \in \{-1, 1\}$ is the label of sample $i$, \rrev{$\tau>0$ } is the regularization parameter, and $\widetilde{Q}$ is a compact convex set. The objective function in \eqref{objective_SVM} is non-differentiable \rev{and because of the existence of the $\ell_2$-norm regularization the value of the Lipschitz constant of such a function can be extremely large.} Thus, we cannot always directly use typical subgradient or gradient schemes to solve the problem \eqref{objective_SVM}. 
The problem of constrained (inequality-type constraints) minimization of convex functions attracts widespread interest in many areas of modern large-scale optimization and its applications \cite{Robust_Truss,Shpirko_2014}. Therefore, we demonstrate the performance of the proposed Algorithm \ref{adaptive_alg3} %and \ref{Alg_50} 
 for such class of problems. We consider an example of the Lagrange saddle point problem induced by the function $f$ in the problem \eqref{objective_SVM}, with some inequality-type function constraints. This problem has the following form
\begin{equation}\label{problem_min_for_SPP}
\min_{x \in  \widetilde{Q}  } \left\{f(x) \left| \; \varphi_p(x):=\sum_{i=1}^n \alpha_{pi}x_i^2 - \beta_p \leqslant 0 \right., \; p=1,...,m \right\},
\end{equation}
where $\rrev{\alpha_{pi}>0}, \beta_p \in \mathbb{R}, \, \forall i=1, \ldots, n$ and $ \forall p = 1, \ldots, m$. The corresponding Lagrange saddle point problem is defined as follows
\begin{equation*}\label{lagrangeSVM}
\min_{x \in \widetilde{Q}} \max_{\boldsymbol{\lambda} = (\lambda_1,\lambda_2,\ldots,\lambda_m)^\top \in \widehat{Q} \subset \mathbb{R}^m_+} \left\{ L(x,\boldsymbol{\lambda}):=f(x)+\sum\limits_{p=1}^m\lambda_p\varphi_p(x)\right\},
\end{equation*}
where $\widehat{Q} $ is a compact convex set. This problem is  equivalent to the variational inequality with the following monotone bounded operator
$$
G(x,\boldsymbol{\lambda})=
\begin{pmatrix}
\nabla f(x)+\sum\limits_{p=1}^m\lambda_p\nabla\varphi_p(x), \\
(-\varphi_1(x),-\varphi_2(x),\ldots,-\varphi_m(x))^\top
\end{pmatrix},
$$
where $ \nabla f$ and $\nabla\varphi_p$ are subgradients of $f$ and $\varphi_p$.

We run Algorithm \ref{adaptive_alg3} with the following prox function
\begin{equation*}\label{prox_function_svm}
d(x, \boldsymbol{\lambda}):=\frac{a_2}{4}\|x\|_{2}^{4}+\frac{a_1}{3}\|x\|_{2}^{3}+\frac{a_0}{2}\|x\|_{2}^{2} + \frac{1}{2}\|\boldsymbol{\lambda}\|_2^2, \quad \forall x \in \mathbb{R}^n, \boldsymbol{\lambda} \in \mathbb{R}_+^m
\end{equation*}
with
\begin{equation}\label{ai_SVN}
a_0 = \frac{1}{n}\sum_{i=1}^{n}\|w_i\|_2^2, \quad a_1 =\frac{2{\moh{\tau}}}{n}\sum_{i=1}^{n}\|w_i\|_2, \quad a_2 = {\moh{\tau}}^2,
\end{equation}
and the following Bregman divergence
\begin{equation}\label{bregman_div_svm}
    V_{\text{new}}\left( \left(y, \boldsymbol{\lambda} \right) ,  \left(x, \boldsymbol{\lambda}^{\prime}\right) \right) = V(y,x) +\frac{1}{2}\|\boldsymbol{\lambda} - \boldsymbol{\lambda}^{\prime}\|_2^2,
\end{equation}
for every $x, y \in \mathbb{R}^n, \boldsymbol{\lambda}, \boldsymbol{\lambda}^{\prime} \in \mathbb{R}_+^m$,  where $V(y,x)$ is given in \eqref{Bregmann_IEP} with coefficients defined in \eqref{ai_SVN}. We consider the ball $\widetilde{Q} \subset \mathbb{R}^n$ at the center $\textbf{0} \in \mathbb{R}^n$ and the radius  $r = \min \left\{\frac{1}{n \tau} \sum\limits_{i=1}^n\left\|w_i\right\|_2, \sqrt{ \frac{2}{\tau} }\right\}$ (see \cite{Lu}). We take the initial point $(x_0, \boldsymbol{\lambda}_0) \in \mathbb{R}^{n+m}$, with all coordinates equaling $0.01$. \rrev{The coefficients $\alpha_{pi}$ in \eqref{problem_min_for_SPP}, and the vectors $w_i$ for $i = 1, \ldots, n$ are chosen randomly from the uniform distribution over $[0, 1)$,} and $\beta_p = r\, (\forall p = 1, \ldots, m)$. We also consider $\widehat{Q} = \{\boldsymbol{\lambda} \in \mathbb{R}^m_+: \|\boldsymbol{\lambda}\|_2^2 \leqslant r^2 \}$, $L_0 = \frac{\|G((1,0,\ldots,0), \textbf{0} ) - G((0,1,0,\ldots,0), \textbf{0} )\|_2}{\sqrt{2}},$ where $\textbf{0} \in \mathbb{R}^m$, $\delta_0 = 0.5,$ and ${\moh{\tau}} = 0.5$ in \eqref{objective_SVM}.
In order to estimate the parameter $R$, for the Bregman divergence \eqref{bregman_div_svm}, we have (see \cite{Lu})
\begin{equation*}
    \begin{aligned}
     V_{\text{new}}\left( \left(x, \boldsymbol{\lambda} \right) ,  \left(x_0, \boldsymbol{\lambda}_0\right) \right) & \leqslant \frac{a_2}{4} \|x-x_0\|_2^2 \left(\|x+x_0\|_2^2 + 2\|x_0\|_2^2\right) 
     \\&\;\;\;\; + \frac{a_1}{3}\|x-x_0\|_2^2 \left(\|x\|_2^2 + 2\|x_0\|_2^2\right) + \frac{a_0}{2}\|x-x_0\|_2^2 
     \\& \;\;\;\; + \frac{1}{2} \|\boldsymbol{\lambda} - \boldsymbol{\lambda}_0 \|_2^2.
    \end{aligned}
\end{equation*}
Therefore in $Q := \widetilde{Q} \times \widehat{Q}$, we have
$$
\begin{aligned}
 &{V_\text{new}}\left( \left(x, \boldsymbol{\lambda} \right) ,  \left(x_0, \boldsymbol{\lambda}_0\right) \right)  \\ &\;\;\;\; \leqslant \left(r + \|x_0\|_2\right)^2 \left[\frac{a_2}{4} \left(r^2 + 2r\|x_0\|_2 + 3 \|x_0\|_2^2\right)  + \frac{a_1}{3} \left(r^2 + 2\|x_0\|_2^2\right) + \frac{a_0}{2}  \right] 
 \\& \;\;\;\; \;\;\;\; + \frac{1}{2} (r + \|\lambda_0\|_2)^2 := v. 
\end{aligned}
$$
Thus we can take $R = \sqrt{v}$. In each iteration of Algorithm \ref{adaptive_alg3},  solving the sub-problem \eqref{minVI}, for the problem \eqref{problem_min_for_SPP}, will be automatically (not explicitly as was for IEP in the previous subsections). 

The results of the work of Algorithm \ref{adaptive_alg3}, for $n = 25, m = 5$, are presented in Fig. \ref{results_alg12_SVM}. These results demonstrate the number of the iterations of Algorithm \ref{adaptive_alg3}, as a function of $\varepsilon \in \{ i^{-1}, i=2, 4, 8, 12, 16, 20\}$. %The results in Fig. \ref{results_alg12_SVM} (in the right) demonstrate the estimate of the achieved solution by Algorithm \ref{Alg_50}, which is in fact the right side of \eqref{estimate_alg50}.
 As it is known, for the variational inequality with a non-smooth operator, the theoretical complexity estimate $O\left(\varepsilon^{-2}\right)$ is optimal. But experimentally we can see from Fig. \ref{results_alg12_SVM} that, the proposed Algorithm \ref{adaptive_alg3}, has iteration complexity nearly to $O\left(\varepsilon^{-1}\right)$, which is an optimal estimate for the problems with smooth operators. %Also, for the considered class of non-smooth problems, we have the quality estimate $O\left(\frac{1}{\sqrt{N}}\right)$ of the solution after $N$ iterations. From the right inset figure in Fig. \ref{results_alg12_SVM}, we see that the proposed Algorithm \ref{Alg_50}, can achieve a better quality estimate when $N$ extremely grows.  

\begin{figure}[htp]
	%\minipage{0.50\textwidth}
	\includegraphics[width=9cm, height=7cm]{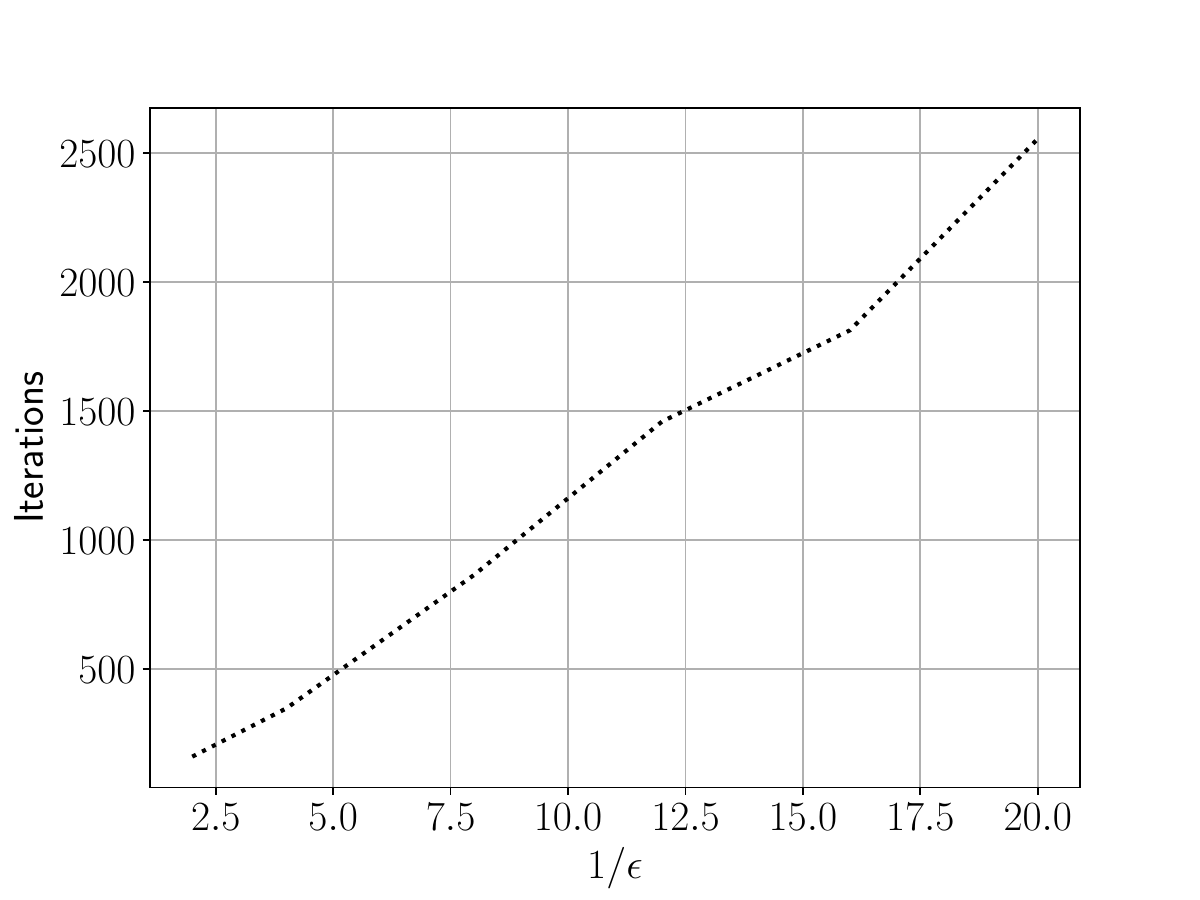}
	%\endminipage\hfill
	%\minipage{0.50\textwidth}
	%\includegraphics[width=\linewidth]{Figs/alg2_svm_new.pdf}
	%\endminipage\hfill
	\caption{The results of Algorithm \ref{adaptive_alg3}, for the problem \ref{problem_min_for_SPP}.}
	\label{results_alg12_SVM}
\end{figure}

\section*{Conclusions}
In this paper, we considered $(\alpha, L, \delta)$-relatively smooth optimization problems which provide for the possibility of minimizing both relatively smooth and relatively Lipschitz continuous functions. For such a type of problems, we introduced some adaptive and universal methods with optimal estimates of the convergence rate. We also considered the problem of solving the variational inequality with a relatively bounded operator.  Finally, we presented the results of numerical experiments for the considered algorithms.

\bigskip
The authors are very grateful to Dmitry Pasechnyuk for fruitful discussions. Also, the authors are very grateful to the unknown reviewers for their extremely valuable comments.

%The work of A. Tytov was supported by the Ministry of Science and Higher Education of the Russian Federation (Goszadaniye), No. 075-00337-20-03, project No. 0714-2020-0005.
%The work of F. Stonyakin and A. Gasnikov was supported by the strategic academic leadership program "Priority 2030" (Agreement  075-02-2021-1316 30.09.2021).
%The work of F. Stonyakin was supported by the strategic academic leadership program "Priority 2030" (Agreement  075-02-2021-1316 30.09.2021).
%The work of A. Gasnikov was supported by the Ministry of Science and Higher Education of the Russian Federation (Goszadaniye), No. 075-00337-20-03, project No. 0714-2020-0005.

\newpage

%%%%%%%%%%%%%%%%%%%%%%%%%%%%%%%%%%%%%%%%%%%%%%%%%%%%%%%%%%%%%%%%%%%%%%%%   Appendix A %%%%%%%%%%%%%%%%%%%%%%%%%%%%%%%%%%%%%%%%%%%%%%%%%%%%
\section*{Appendix A. The proof of Theorem \ref{the_VI}}
\begin{proof}
Due to \eqref{minVI}, \eqref{Lemma} \rev{and \eqref{Lemma}}, for each $x \in Q$, we have 
$$
\begin{aligned}
\langle g(x_k),x_{k+1}-x\rangle  \leqslant L_{k+1}V(x,x_k)-L_{k+1}V(x,x_{k+1})- L_{k+1}V(x_{k+1},x_k).
\end{aligned}
$$

Thus, taking into account \eqref{condVI} and monotonicity of $g$, we get
\begin{equation*}
    \begin{aligned}
     L_{k+1}V(x,x_k)-L_{k+1}  V(x&,x_{k+1}) \geqslant \langle g(x_k),x_{k+1}-x\rangle+L_{k+1}V(x_{k+1},x_k)
     \\& = \langle g(x_k),x_{k+1}-x\rangle+\langle g(x_k),x_{k+1}-x_k\rangle 
     \\& \;\;\;\; + L_{k+1}V(x_{k+1},x_k) - \langle g(x_k),x_{k+1}-x_k\rangle.
     \\&  \geqslant\langle g(x_k),x_{k+1}-x\rangle-\langle g(x_k),x_{k+1}-x_k\rangle-\frac{\varepsilon}{2}
     \\& = \langle g(x_k),x_{k}-x \rangle - \frac{\varepsilon}{2} \geqslant \langle g(x_k),x_{k}-x \rangle - \frac{\varepsilon}{2},
    \end{aligned}
\end{equation*}
whence we obtain that
\begin{equation}\label{eee}
    \langle g(x),x_{k}-x\rangle\leqslant L_{k+1}V(x,x_k)-L_{k+1}V(x,x_{k+1})+\frac{\varepsilon}{2}, \quad \forall x \in Q.
\end{equation}
Taking summation over both sides of \eqref{eee}, we have
$$
    \sum\limits_{k=0}^{N-1} \frac{1}{L_{k+1}}\langle g(x),x_{k}-x\rangle\leqslant V(x,x_0)+ \sum\limits_{k=0}^{N-1}\frac{\varepsilon}{2L_{k+1}}, \quad \forall x\in Q,
$$
which leads to
$$
    \langle g(x),\widehat{x}-x\rangle\leqslant \frac{1}{S_N}\sum\limits_{k=0}^{N-1} \frac{1}{L_{k+1}}\langle g(x),x_{k}-x\rangle\leqslant \frac{R^2}{S_N}+\frac{\varepsilon}{2},
$$
and
$$
    \max\limits_{x\in Q}\langle g(x),\widehat{x}-x\rangle\leqslant \frac{R^2}{S_N}+\frac{\varepsilon}{2}.
$$

As operator $g$ is relatively bounded, i.e.
$$
    \langle g(x), x - y\rangle \leqslant M\sqrt{2V(y,x)} \leqslant \frac{M^2V(y,x)}{\varepsilon} + \frac{\varepsilon}{2}, \quad \forall \varepsilon >0,
$$
the stopping criterion of Algorithm \ref{adaptive_alg3} is guaranteed to be satisfied for $L_{k+1}\geqslant \frac{M^2}{\varepsilon}$. 
Since the exit from the iteration will certainly happen for $L_{k+1} \leqslant \frac{2M^2}{\varepsilon},$ we have
$$
    \frac{R^2}{S_N} \leqslant \frac{2M^2R^2}{\varepsilon N}.
$$
Thus, the total number of iterations of Algorithm \ref{adaptive_alg3} will not exceed 
$$
N=\left\lceil\displaystyle\frac{4M^2R^2}{\varepsilon^2}\right\rceil.
$$
\end{proof}

\section*{Appendix B. The proof of Theorem \ref{theorem_estimate_alg50}}

\begin{proof}
The proof of \eqref{estimate_alg50} is similar to the proof of Theorem \ref{the_VI}, with  $\displaystyle\frac{\varepsilon}{2} = \delta_{k+1}.$
Let us assume that on the  $(k+1)$-th iteration  $(k=0,1,\ldots, N-1)$ of the Algorithm \ref{Alg_50}, the auxiliary problem \eqref{subproblem_alg50} is solved  $i_{k+1}$ times. Then
$$ 
    2^{i_{k+1}-2}= \frac{L_{k+1}}{L_{k}}=\frac{\delta_{k+1}}{\delta_{k}},
$$
since at the beginning of each iteration the parameters  $L_{k}, \delta_{k}$ are divided by 2. Therefore,
$$
    \sum\limits_{k=0}^{N-1} i_{k+1}=2N+\log_2 \frac{L_N}{L_0},\quad \log_2 \frac{L_N}{L_0}=\log_2 \frac{\delta_N}{\delta_0}.
$$
It is clear that at least one of the inequalities  $L_N \leqslant 2L, \delta_N \leqslant 2 \delta$ holds, which ends the proof of the theorem.
\end{proof}

\section*{Appendix C. The proof of Theorem \ref{theorem_adaptive_Alg_2}}
\begin{proof}
Let us use the reasoning in the proof of Theorem \ref{the_VI} for $ g(x)=\nabla f(x)$. \rev{Taking into account \eqref{Lemma},} for any $x \in Q$, we have, 
\begin{equation*}
    \begin{aligned}
     \langle \nabla f(x_k),x_{k+1}-x\rangle \leqslant L_{k+1}V(x,x_k)-L_{k+1}V(x,x_{k+1})-L_{k+1}V(x_{k+1},x_k).
    \end{aligned}
\end{equation*}
%Further,
%$$
%    \langle\nabla f(x_k),x_{k+1}-x\rangle\leqslant L_{k+1}V(x,x_k)-L_{k+1}V(x,x_{k+1})-L_{k+1}V(x_{k+1},x_k),
%$$
Thus, taking into account \eqref{condalg4}, we get
\begin{equation*}
    \begin{aligned}
     L_{k+1}V(x,x_k)-L_{k+1}V( & x,  x_{k+1})  \geqslant \langle\nabla f(x_k),x_{k+1}-x\rangle+L_{k+1}V(x_{k+1},x_k)
     \\& = \langle\nabla f(x_k),x_{k+1}-x\rangle + \langle\nabla f(x_k),x_{k+1}-x_k\rangle
     \\& \;\;\;\; + L_{k+1}V(x_{k+1},x_k)-\langle\nabla f(x_k),x_{k+1}-x_k\rangle
     \\& \geqslant\langle\nabla f(x_k),x_{k+1}-x\rangle-\langle\nabla f(x_k),x_{k+1}-x_k\rangle-\frac{\varepsilon}{2}
     \\& =\langle\nabla f(x_k),x_{k}-x \rangle - \frac{\varepsilon}{2}.
    \end{aligned}
\end{equation*}
So, we have
\begin{equation}\label{eq1}
    \langle \nabla f(x_k),x_{k}-x\rangle\leqslant L_{k+1}V(x,x_k)-L_{k+1}V(x,x_{k+1})+\frac{\varepsilon}{2}, \quad \forall x \in Q.
\end{equation}
Taking summation over both sides of \eqref{eq1}, we obtain
$$
    \sum\limits_{k=0}^{N-1} \frac{1}{L_{k+1}}\langle \nabla f(x_k),x_{k}-x\rangle\leqslant V(x,x_0)+\sum\limits_{k=0}^{N-1}\frac{\varepsilon}{2L_{k+1}},\quad \forall x\in Q.
$$
Further, in view of the inequality
$$
    \langle \nabla f(x_k),x_{k}-x\rangle\geqslant f(x_k)-f(x),
$$
we have
$$
    \sum\limits_{k=0}^{N-1} \frac{1}{L_{k+1}}\left(f(x_k)-f(x)\right)\leqslant V(x,x_0)+\sum\limits_{k=0}^{N-1}\frac{\varepsilon}{2L_{k+1}}.
$$
Moreover, since $f$ is convex, the following inequality holds
$$
    \rev{\sum\limits_{k=0}^{N-1} \frac{1}{L_{k+1}}} f(x_k)\geqslant S_Nf(\widehat{x}),
$$
where $S_N=\sum\limits_{k=0}^{N-1} \frac{1}{L_{k+1}}.$
Then
\begin{equation}\label{eq2}
    \rev{\sum\limits_{k=0}^{N-1} \frac{1}{L_{k+1}}}\left(f(x_k)-f(x)\right)\leqslant V(x,x_0) + \frac{\varepsilon}{2}S_N.
\end{equation}
Since $V(x_*,x_0)\leqslant R^2,$ we obtain, for  $x = x_*$ in \eqref{eq2}, that
$$
    f(\widehat{x})-f(x_*)\leqslant \frac{R^2}{S_N}+\frac{\varepsilon}{2}.
$$
\end{proof}
\section*{Appendix D. The proof of Theorem \ref{theorem_Algor2}}
\begin{proof}
1) Taking into account the standard minimum condition for the subproblem \eqref{eqproblem} \rev{and \eqref{Lemma}}, we have
$$
    \langle \nabla f(x_k) + L_{k+1}\nabla_{x = x_{k+1}}V(x, x_k), x - x_k \rangle \geqslant 0, \quad \forall\, x \in Q.
$$
After the completion of the  $k$-th iteration  $(k = 0, 1, \ldots)$ of the Algorithm  \ref{Algor2}, the following inequalities hold
\begin{equation*}
    \begin{aligned}
     \langle\nabla f(x_k),x_{k+1}-x \rangle \leqslant L_{k+1}V(x,x_k)-L_{k+1}V(x,x_{k+1})-L_{k+1}V(x_{k+1},x_k),
    \end{aligned}
\end{equation*}
and
$$
    f(x_{k+1})\leqslant f(x_{k})+ \langle\nabla f(x_k),x_{k+1}-x_k\rangle+L_{k+1}V(x,x_k)-L_{k+1}V(x,x_{k+1})+\delta_{k+1}.
$$
Therefore,
$$
    f(x_{k+1})\leqslant f(x_{k})+ \langle\nabla f(x_k),x-x_k\rangle + L_{k+1}V(x,x_k)-L_{k+1}V(x,x_{k+1})+\delta_{k+1}.
$$
Further, taking into account the inequality $f(x_{k})+ \langle\nabla f(x_k),x-x_k\rangle \leqslant f(x)$, for $x=x_{*}$, we obtain
$$
 f(x_{k+1})-f(x_{*})\leqslant L_{k+1}V(x_{*},x_k)-L_{k+1}V(x_{*},x_{k+1})+\delta_{k+1},
$$
whence, after summation, in view of the convexity of  $f$, we have
\color{black}{
$$
    f(\widehat{x})-f(x_{*})\leqslant\frac{1}{S_N}\sum\limits_{k=0}^{N-1} \frac{f(x_{k+1})}{L_{k+1}}-f(x_{*})\leqslant \frac{V(x_{*},x_0)}{S_N}+ \frac{1}{S_N}\sum\limits_{k=0}^{N-1} \frac{\delta_{k+1}}{L_{k+1}}.
$$}
2) Since $f$ satisfies \eqref{eqalpha1relsm} and \eqref{eqalpha1relsm1}, for sufficiently large $L_{k+1}$ and $\delta_{k+1}$ the iteration exit criterion will certainly be satisfied. According to \eqref{eqalpha1relsm}, for some fixed  $L>0$ and $\delta > 0$, the following inequality holds
$$
    f(y) \leqslant f(x) + \langle \nabla f(x), y - x \rangle + LV(y, x) + \alpha LV(x, y) + \delta, \quad \forall x, y \in Q.
$$
Therefore, for  $L_{k+1} \geqslant L$ and taking into account  \eqref{eqproblem} we obtain
\color{black}{\begin{equation*}
    \begin{aligned}
    f(x_{k+1})&\leqslant f(x_k) + \langle \nabla f(x_k), x_{k+1} - x_k \rangle + L_{k+1}(V(x_{k+1}, x_k) 
    \\& \;\;\;\; + \alpha V(x_k, x_{k+1}))+ \delta 
    \\& \leqslant f(x_k) - L_{k+1}(1 - \alpha)V(x_k, x_{k+1}) + \delta,
    \end{aligned}
\end{equation*}
}
whence
\begin{equation}\label{ineq1}
    \alpha f(x_{k+1})\leqslant \alpha f(x_k) - L_{k+1}\alpha(1 - \alpha)V(x_k, x_{k+1}) + \alpha\delta.
\end{equation}
Now, in view of  \eqref{eqalpha1relsm} and taking into account  $1 - \alpha \geqslant 0$, the following inequality holds
\begin{equation}\label{ineq2}
\begin{aligned}
(1-\alpha)f(x_{k+1}) &\leqslant (1-\alpha) f(x_k) + (1-\alpha)\langle \nabla f(x_k), x_{k+1} - x_k \rangle + 
\\& \;\;\;\; + L_{k+1}(1- \alpha)(V(x_{k+1}, x_{k}) + \alpha V(x_{k}, x_{k+1})) + (1-\alpha)\delta
\end{aligned}
\end{equation}
and for $\alpha = 0$ we have
\begin{equation}\label{ineq211}
\begin{aligned}
 f(x_{k+1})\leqslant f(x_k) + \langle \nabla f(x_k), x_{k+1} - x_k \rangle + L_{k+1}V(x_{k+1}, x_{k}) + \delta.
\end{aligned}
\end{equation}
Taking into account \eqref{eqalpha1relsm1} for $\alpha >0$, after summing the inequalities \eqref{ineq1} and \eqref{ineq2}, we have
\begin{equation*}\label{ineq3}
\begin{aligned}
f(x_{k+1}) & \leqslant f(x_k) + (1 - \alpha)\langle \nabla f(x_k), x_{k+1} - x_k \rangle +  L_{k+1}(1 - \alpha)V(x_{k+1}, x_{k}) + 
\\&  \;\;\;\; \delta
\\& \leqslant f(x_k) + \langle \nabla f(x_k), x_{k+1} - x_k \rangle + L_{k+1}V(x_{k+1}, x_{k}) + \alpha\delta 
\\&\leqslant f(x_k) + \langle \nabla f(x_k), x_{k+1} - x_k \rangle + L_{k+1}V(x_{k+1}, x_{k}) + \delta,
\end{aligned}
\end{equation*}
i.e. \eqref{ineq211} holds for each  $\alpha \in [0; 1]$. It means that the iteration exit criterion of the Algorithm \ref{Algor2} will certainly  be satisfied for  $L_{k+1} \geqslant L$ and $\delta_{k+1} \geqslant \delta$.

3) Let us assume that on the  $(k+1)$-th iteration  $(k=0,1,\ldots, N-1)$ of the Algorithm \ref{Algor2}, the auxiliary problem \eqref{eqproblem} is solved  $i_{k+1}$ times. Then
$$ 
    2^{i_{k+1}-2}= \frac{L_{k+1}}{L_{k}}=\frac{\delta_{k+1}}{\delta_{k}},
$$
since at the beginning of each iteration the parameters  $L_{k}, \delta_{k}$ are divided by 2. Therefore,
$$
    \sum\limits_{k=0}^{N-1} i_{k+1}=2N+\log_2 \frac{L_N}{L_0},\quad \log_2 \frac{L_N}{L_0}=\log_2 \frac{\delta_N}{\delta_0}.
$$
It is clear that at least one of the inequalities  $L_N\leqslant2L, \delta_N\leqslant2\delta$ holds, which ends the proof.
\end{proof}

\section*{Appendix E. The proof of Theorem \ref{ThmUnivMeth2}}
\begin{proof}

\color{black}{1) Analogously with i.1) of the  Theorem's \ref{theorem_Algor2} proof, we have
\begin{equation}\label{equat5.2}
      f(\widehat{x})-f(x_{*})\leqslant\frac{1}{S_N}\sum\limits_{k=0}^{N-1} \frac{f(x_{k+1})}{L_{k+1}}-f(x_{*})\leqslant \frac{V(x_{*},x_0)}{S_N}+  \frac{3\varepsilon}{4}.
\end{equation}
}

2) Analogously with i.2) of the Theorem's \ref{theorem_Algor2} proof, we conclude that for each ($\alpha, L, \delta$)-relatively smooth function $f$ the criterion for the exit from the iteration is certainly fulfilled for $L_{k+1}\geqslant L$.

%A gradient-type descent step of the form

%\begin{equation*}\label{plus_ineq}
%    x_{k+1}:= \arg\min\limits_{x\in Q}\{\langle\nabla f(x_k),x\rangle+L_{k+1}V(x,x_k)\},
%\end{equation*}
%means that the following inequality holds
%\begin{equation}\label{2star_ineq}
%    \begin{aligned}
%     \langle\nabla f(x_k),x_{k+1}-x_k\rangle &\leqslant -L_{k+1}\langle\nabla d(x_{k+1})-\nabla d(x_k),x_{k+1} -x_k\rangle \\
%     & =-L_{k+1}V(x_{k+1},x_k)-L_{k+1}V(x_{k},x_{k+1}).
%    \end{aligned}
%\end{equation}

%Due to $M$-relative Lipschitz continuity of $f$
%$$
%    \langle \nabla f(x),y-x\rangle+M\sqrt{2V(y,x)}\geqslant 0, \quad \forall x,y\in Q,
%$$
3) Due to \eqref{equat5.2} for each $k \geqslant 0$ we have
%$$
%    0\leqslant \langle\nabla f(x),y-x\rangle + M \sqrt{2V(y,x)} \leqslant \langle\nabla f(x),y-x\rangle + \frac{2M^2}{\varepsilon}V(y,x) +\frac{\varepsilon}{4}.
%$$
%and
%$$
%    f(x_{k+1})\leqslant f(x_k)+ \langle \nabla f(x_k),x_{k+1}-x_k\rangle  + L_{k+1}V(x_{k+1},x_k) + \frac{3\varepsilon}{4}.
%$$
%and for $L_{k+1}\geqslant L$ the criterion for the exit from the iteration is certainly fulfilled. 

%This means that we can apply Theorem 3.1 from \cite{Stonyakin_etal} about the evaluation of the quality of the solution produced by the gradient method using the $\left(\delta,L,V \right)$-model $(\delta=\frac{3\varepsilon}{4})$ at the requested point. So, for $S_N =\sum\limits_{k=0}^{N-1} \frac{1}{L_{k+1}}$ we obtain
$$
    f(\widehat{x})-f(x_*)\leqslant\frac{R^2}{S_N}+\frac{3\varepsilon}{4}.
$$
So, for each ($\alpha, L, \delta$)-relatively smooth function $f$ the exit from the iteration will certainly happen for  $L_{k+1}\leqslant 2L,$ whence
$$
    S_N\geqslant \frac{N}{2L}\quad \text{and}\quad \frac{R^2}{S_N}\leqslant\frac{2LR^2}{N}.
$$    
If we require the condition $\displaystyle\frac{2LR^2}{N}\leqslant\displaystyle\frac{\varepsilon}{4},$ we have that $f(\widehat{x})-f(x_*)\leqslant \varepsilon$  certainly holds for $$N\geqslant \frac{8LR^2}{\varepsilon}.$$

%The exit from the iteration will certainly happen for  $L_{k+1}\leqslant \displaystyle\frac{4M^2}{\varepsilon},$ whence
%$$
%    S_N\geqslant \frac{\varepsilon N}{4M^2}\quad \text{and}\quad \frac{R^2}{S_N}\leqslant\frac{4M^2R^2}{\varepsilon N}.
%
If we require that f is a $\left(1, \frac{2M^2}{\varepsilon}, \frac{\varepsilon}{2}\right)$-relatively smooth function,
we have that $$N\geqslant \frac{8LR^2}{\varepsilon} = \frac{16M^2R^2}{\varepsilon^2}.$$

%$f(\widehat{x})-f(x_*)\leqslant \varepsilon$  certainly holds for $N\geqslant \displaystyle\frac{16M^2R^2}{\varepsilon^2}$.

Note, that if $f$ is relatively Lipschitz continuous function then f is a $\left(1, \frac{2M^2}{\varepsilon}, \frac{\varepsilon}{2}\right)$-relatively smooth function. Indeed, we have
\textcolor{black}{$$
  f(y) - f(x)\leqslant \langle \nabla f(y),y - x \rangle \leqslant M\sqrt{2V(x,y)},
$$}
i.e.
\begin{equation}\label{111}
    f(y) \leqslant f(x) + M\sqrt{2V(x,y)},
\end{equation}
and
\begin{equation}\label{222}
    \langle\nabla f(x),x-y\rangle 
    \color{black}{\leqslant} M\sqrt{2V(y,x)}.
\end{equation}

After summing inequalities \eqref{111} and \eqref{222}, we get, that the following inequalities hold for any $x, y \in Q$:
\begin{equation}\label{1star_ineq}
   \begin{aligned}
    f(y) &\leqslant f(x)+\langle\nabla f(x),y-x\rangle +M \left(\sqrt{2V(y,x)}+\sqrt{2V(x,y)}\right) \\
   &\leqslant f(x)+\langle \nabla f(x),y-x\rangle +\frac{2M^2}{\varepsilon}\left(V(y,x)+V(x,y)\right)+ \frac{\varepsilon}{2}.
   \end{aligned}
\end{equation}
%Let us note, that the following inequality holds for any $a, b \geqslant 0$
%\begin{equation*}\label{3star_ineq}
%  \sqrt{2ab}\leqslant a+\frac{b}{2}, \quad \forall a,b\geqslant 0.  
%\end{equation*}

%Inequalities \eqref{1star_ineq} mean that for $\displaystyle L_{k+1} \geqslant \frac{\displaystyle2M^2}{\displaystyle\varepsilon}$
%, after step \eqref{plus_ineq} is certainly true, we have the following
%$$
%    f(x_{k+1})\leqslant f(x_k)+\langle\nabla f(x_k),x_{k+1}-x_k\rangle + L_{k+1}V(x_{k+1},x_k) +L_{k+1}V(x_{k},x_{k+1}) + \frac{\varepsilon}{2},
%$$
%whence, taking into account \eqref{2star_ineq}, we obtain
%\begin{equation*}\label{2plus_ineq}
%    f(x_{k+1})\leqslant f(x_k)+ \frac{\varepsilon}{2}.
%\end{equation*}
\end{proof}


\begin{thebibliography}{99}

\bibitem{AdaMirr_2021}
     K. Antonakopoulos and  P. Mertikopoulos, {\it Adaptive first-order methods revisited: Convex optimization without Lipschitz requirements}, NeurIPS (2021),  \url{https://arxiv.org/pdf/2107.08011.pdf}
     
     \bibitem{Antonakopoulos}
    K. Antonakopoulos, E. V. Belmega and  P. Mertikopoulos, {\it An adaptive mirror-prox algorithm for variational inequalities with singular operators}, In NeurIPS (2019)
    
    
    \bibitem{Bauschke}
	H. H. Bauschke, J. Bolte and M. Teboulle,  {\it A descent lemma beyond Lipschitz gradient continuity: first-order methods revisited and applications}, Mathematics of Operations Research  {\bf 42}(2)  (2017), 330--348.
	
	\bibitem{Robust_Truss}
    A. Ben-Tal and A. Nemirovski, {\it Robust Truss Topology Design via Semidefinite Programming}, SIAM J. Optim. \textbf{7}(4), 991--1016 (1997)
	
	%\bibitem{Cohen} %%%%%%%%%%
    %Cohen, M. B., Sidford, A.,  Tian, K.: Relative Lipschitzness in Extragradient Methods and a Direct Recipe for Acceleration. (2020). \url{https://arxiv.org/pdf/2011.06572.pdf}
    
    \bibitem{Dragomir}
    R. A. Dragomir, A. B. Taylor, A. d’Aspremont and  J. Bolte, {\it Optimal complexity and certification of Bregman first-order methods}, Mathematical Programming (2021), 1--43.
	
	\bibitem{Lu_Nesterov}
	H. Lu, R. Freund and  Yu. Nesterov, {\it Relatively smooth convex optimization by first-order methods, and applications}, SIOPT {\bf 28}(1) (2018), 333--354.
	
	\bibitem{kamzolov2022exploiting}
	D. Kamzolov, A. Gasnikov, P. Dvurechensky, A. Agafonov, M. Takáč {\it Exploiting higher-order derivatives in convex optimization methods} (2022) \url{https://arxiv.org/pdf/2208.13190.pdf}
	
	
	%\bibitem{Gasnikov_opt}
	%A. V. Gasnikov, E. A. Gorbunov, D. A. Kovalev, A. A. M. Mohammed and E. O. Chernousova, {\it Substantiation of the hypothesis about optimal estimates of the rate of convergence of numerical methods of high-order convex optimization}, Computer Research and Modeling {\bf 10}(6) (2018), 737--753.
	
	\bibitem{Lu}
	H. Lu, {\it Relative Continuity for Non-Lipschitz Nonsmooth Convex Optimization Using Stochastic (or Deterministic) Mirror Descent}, Informs Journal on Optimization {\bf 1}(4) (2019), 288--303.
	
	\bibitem{Nest_tens}
	Y. Nesterov, {\it Implementable tensor methods in unconstrained convex optimization}, Mathematical Programming (2019), 1--27.
	
	\bibitem{Nest_core}
	Y. Nesterov, {\it Inexact accelerated high-order proximal-point methods}, (No. UCL-Université Catholique de Louvain) CORE. (2020) 
	
	
	
	\bibitem{Nestconf}
	Y. Nesterov, {\it Relative Smoothness: New Paradigm in Convex Optimization}, Conference report, EUSIPCO-2019, A Coruna, Spain, September 4, (2019)
    
    \bibitem{Stonyakin_etal}
    F. Stonyakin, A. Tyurin, A. Gasnikov, P. Dvurechensky, A. Agafonov, D. Dvinskikh,  M. Alkousa, D. Pasechnyuk,  S. Artamonov, and V. Piskunova, {\it Inexact relative smoothness and strong convexity for optimization and variational inequalities by inexact model},  Optim. Methods and Software, {\bf 36}(6) (2021), 1155--1201.
    
    \bibitem{pegasos_2011}
    S. S. Shwartz, Y. Singer,  N. Srebro and A. Cotter, {\it Pegasos: primal estimated sub-gradient solver for SVM}, Mathematical Programming, {\bf 127} (2011), 3--30.
    
    \bibitem{Shpirko_2014}
    S. Shpirko and Yu. Nesterov, {\it Primal-dual subgradient methods for huge-scale linear conic problem}, SIAM J. Optim. \textbf{24}(3) (2014), 1444--1457.
    
    \bibitem{savchuk2022adaptive}
    O. S. Savchuk, A. A. Titov, F. S. Stonyakin and M. S. Alkousa, {\it Adaptive first-order methods for relatively strongly convex optimization problems}, Computer Research and Modeling, {\bf 14}(2) (2022), 445--472.
    
    \bibitem{Titov_etal}
    A. A. Titov, F. S. Stonyakin,  M. S. Alkousa, S. S. Ablaev and   A. V. Gasnikov,  {\it Analogues of Switching Subgradient Schemes for Relatively Lipschitz-Continuous Convex Programming Problems}, In International Conference on Mathematical Optimization Theory and Operations Research, Springer, Cham. (2020) 133--149.
    
    \bibitem{MOTOR2021}
    A. Titov, F. Stonyakin, M. Alkousa and A. Gasnikov, {\it Algorithms for solving variational inequalities and saddle point problems with some generalizations of Lipschitz property for operators}, Communications in Computer and Information Science,  Springer, Cham. {\bf 1476}  (2021) 86--101. %\url{https://arxiv.org/pdf/2103.00961.pdf}
   
    %\bibitem{titov2020online}
    %A. A. Titov, F. S. Stonyakin, A. V. Gasnikov, M. S. Alkousa: Mirror Descent and Constrained Online Optimization Problems. Communications in Computer and Information Science, 974, pp. 64-78.
    
    %\bibitem{article:hazan_beyond_2014}
	%E.~Hazan, S.~Kale: Beyond the regret minimization barrier: Optimal algorithms for stochastic strongly-convex optimization. JMLR. \textbf{15}, pp. 2489--2512, 2014.

    %\bibitem{article:hazan_introduction_2016}
	%E.~Hazan: Introduction to online convex optimization. Second edition, 2021. \url{https://arxiv.org/pdf/1909.05207.pdf}

    %\bibitem{paper:Yuan_online_cumulative_cons}
	%J.~Yuan, A.~Lamperski: Online convex optimization for cumulative constraints. Published in NIPS, pp.~6140--6149, 2018.

    %\bibitem{article:zinkevich_online_2003}
	%M.~Zinkevich: Online Convex Programming and Generalized Infinitesimal Gradient Ascent.	In Proceedings of International Conference on Machine Learning (ICML), 2003.
	
    %\bibitem{article:bubeck_eldan_2016}
	%S.~Bubeck, R.~Eldan: Multi-scale exploration of convex functions and bandit convex optimization. JMLR: Workshop and Conference Proceedings \textbf{49}, pp. 1--7, 2016.	
	
    %\bibitem{article:bubeck_bianchi_2012}
	%S.~Bubeck, N.~Cesa-Bianchi: Regret analysis of stochastic and nonstochastic multi-armed bandit problems. Foundation and Trends in Machine Learning, \textbf{5}(1), pp. 1--122, 2012.	
	
    %\bibitem{article:gasnikov_stoc_online_2017}
	%A.~V.~Gasnikov, A.~A.~Lagunovskaya, I.~N.~Usmanova, F.~A.~Fedorenko, E.~A.~Krymova: Stochastic online optimization. Single-point and multi-point non-linear multi-armed bandits. Convex and strongly-convex case. Automation and Remote Control, \textbf{78}(2), pp. 224--234, 2017.
	
    %\bibitem{orabona2022}
    %Francesco Orabona: A Modern Introduction to Online Learning. 2022. \url{https://arxiv.org/pdf/1912.13213v1.pdf}
    
    %\bibitem{online_stoc_moh}
    %M. S. Alkousa: On Some Stochastic Mirror Descent Methods for Constrained Online Optimization Problems. Computer Research and Modeling, 11(2), pp. 205–-217, 2019. \url{http://crm-en.ics.org.ru/journal/article/2775/}


    %\bibitem{article:awerbuch_2008}
	%B.~Awerbuch, R.~Kleinberg: Online linear optimization and adaptive routing. Journal of Computer and System Sciences. \textbf{74}(1), pp.~97--114, 2008.

    %\bibitem{article:jenatton_2015}
    %R.~Jenatton, J.~Huang, C.~Archambeau: Adaptive Algorithms for Online Convex Optimization with Long-term Constraints.  Proceedings of The 33rd International Conference on Machine Learning, PMLR 48, pp. 402--411, 2016. \url{http://proceedings.mlr.press/v48/jenatton16.pdf}.


\end{thebibliography}
\end{document}